\documentclass{amsart}
\usepackage{hyperref}
\usepackage{amsrefs}
\usepackage{tikz}
\usepackage{dsfont}
\usepackage{newverbs}
\usepackage{fancyvrb}
\usepackage{comment}
\usepackage{mathrsfs}
\usepackage{mathtools}
\usepackage{amssymb}

\usepackage{tikz}
\usetikzlibrary{matrix,arrows,decorations.pathmorphing, cd}

\newif\ifshowkeys
\showkeysfalse

\ifshowkeys
\newcommand{\lbl}[1]{\label{#1}\textup{[\texttt{#1}]}\par}
\else
\newcommand{\lbl}{\label}
\fi

\newcommand{\Ab}	{\operatorname{Ab}}
\newcommand{\CAT}	{\operatorname{CAT}}
\newcommand{\DER}	{\operatorname{DER}}
\newcommand{\End}	{\operatorname{End}}
\newcommand{\Ho}	{\operatorname{Ho}}
\newcommand{\Map}       {\operatorname{Map}}
\newcommand{\POSet}	{\operatorname{POSet}}
\newcommand{\Top}       {\operatorname{Top}}

\newcommand{\cof}       {\operatorname{cof}}
\newcommand{\fib}       {\operatorname{fib}}
\newcommand{\inc}       {\operatorname{inc}}
\newcommand{\opp}	{{\operatorname{op}}}
\newcommand{\supp}	{\operatorname{supp}}
\newcommand{\tfib}	{\operatorname{tfib}}

\newcommand{\LA}	{{\mathbb{A}}}

\newcommand{\MM}        {{\mathbb{M}}}
\newcommand{\PP}        {{\mathbb{P}}}
\newcommand{\QQ}        {{\mathbb{Q}}}

\newcommand{\CA}        {{\mathcal{A}}}
\newcommand{\CB}        {{\mathcal{B}}}
\newcommand{\CC}        {{\mathcal{C}}}
\newcommand{\CD}        {{\mathcal{D}}}
\newcommand{\CE}        {{\mathcal{E}}}
\newcommand{\CF}        {{\mathcal{F}}}
\newcommand{\CG}        {{\mathcal{G}}}

\newcommand{\CJ}        {{\mathcal{J}}}
\newcommand{\CL}        {{\mathcal{L}}}
\newcommand{\CP}        {{\mathcal{P}}}
\newcommand{\CT}        {{\mathcal{T}}}
\newcommand{\CU}        {{\mathcal{U}}}
\newcommand{\CX}        {{\mathcal{X}}}
\newcommand{\CY}        {{\mathcal{Y}}}

\newcommand{\Gm}        {\Gamma}
\newcommand{\Sg}        {\Sigma}

\newcommand{\Lm}        {\Lambda}
\newcommand{\Om}        {\Omega}

\newcommand{\tSg}       {\widetilde{\Sigma}}

\newcommand{\al}        {\alpha}
\newcommand{\bt}        {\beta} 
\newcommand{\gm}        {\gamma} 
\newcommand{\dl}        {\delta}

\newcommand{\kp}        {\kappa}

\newcommand{\tht}       {\theta}
\newcommand{\lm}        {\lambda}
\newcommand{\sg}        {\sigma}

\newcommand{\N}         {{\mathbb{N}}}
\newcommand{\Z}         {{\mathbb{Z}}}
\newcommand{\Q}         {{\mathbb{Q}}}
\newcommand{\Zpl}       {{\mathbb{Z}_{(p)}}}

\newcommand{\Wedge}     {\vee}
\newcommand{\Smash}     {\wedge}
\newcommand{\bigWedge}  {\bigvee}
\newcommand{\sm}        {\setminus}
\newcommand{\sse}       {\subseteq}
\newcommand{\st}        {\;|\;}
\newcommand{\bc}[1]	{\langle #1\rangle}

\newcommand{\tm}        {\times}
\newcommand{\hK}	{\widehat{K}}

\newcommand{\ostar}     {\boxtimes}

\newcommand{\xla}       {\xleftarrow}
\newcommand{\xra}       {\xrightarrow}

\newcommand{\invlim} {\operatornamewithlimits{\underset{\longleftarrow}{lim}}}
\newcommand{\holim}  {\operatornamewithlimits{\underset{\longleftarrow}{holim}}}
\newcommand{\hocolim}  {\operatornamewithlimits{\underset{\longrightarrow}{holim}}}

\renewcommand{\:}{\colon}

\newtheorem{theorem}{Theorem}[section]

\newtheorem{lemma}[theorem]{Lemma}
\newtheorem{proposition}[theorem]{Proposition}
\newtheorem{corollary}[theorem]{Corollary}
\theoremstyle{definition}

\newtheorem{remark}[theorem]{Remark}
\newtheorem{definition}[theorem]{Definition}
\newtheorem{example}[theorem]{Example}
\newtheorem{construction}[theorem]{Construction}

\DefineVerbatimEnvironment{checks}{Verbatim}{formatcom=\color{blue}}
\DefineVerbatimEnvironment{mcodeblock}{Verbatim}{formatcom=\color{darkred}}
\definecolor{darkred}{rgb}{0.5,0,0}
\definecolor{darkgreen}{rgb}{0,0.5,0}
\newverbcommand{\mcode}{\color{darkred}}{}
\newverbcommand{\fname}{\color{darkgreen}}{}

\begin{document}
\title{Iterated chromatic localisation}
\author{N.~P.~Strickland}
\author{N.~Bellumat}

\maketitle 

\begin{abstract}
 We study a certain monoid of endofunctors of the stable homotopy
 category that includes localizations with respect to finite unions of
 Morava $K$-theories.  We work in an axiomatic framework that can also
 be applied to analogous questions in equivariant stable homotopy
 theory.  Our results should be helpful for the study of
 transchromatic phenomena, including the Chromatic Splitting
 Conjecture.  The combinatorial parts of this work have been
 formalised in the Lean proof assistant.
\end{abstract}

\section{Introduction}
\label{sec-intro}

Fix a prime $p$, and let $\CB$ denote the category of $p$-local
spectra.

The Bousfield localisation functors $L_{K(n)}:\CB\to\CB$ and
$L_n=L_{K(0)\Wedge\dotsb\Wedge K(n)}$ play a central role in chromatic
homotopy theory.  It is a well-known and useful fact that
$L_nL_m=L_{\min(n,m)}$ and $L_nL_{K(n)}=L_{K(n)}$.  It is not hard to
see that $L_{K(n)}L_{K(m)}=L_{K(n)}L_m=0$ when $n>m$.  Two versions of
the Chromatic Splitting Conjecture of Hopkins involve the functors
$L_{n-1}L_{K(n)}$ and $L_{K(n-1)}L_{K(n)}$, and the latter can
naturally be compared with $L_{K(n-1)\Wedge K(n)}$.  Spectra such as
$L_{K(m)}\widehat{E(n)}=L_{K(m)}L_{K(n)}E(n)$ (for $m<n$) occur
naturally in the transchromatic character theory of Stapleton, and
also in work of Torii.  To encompass all these examples, we make the
following definitions:
\begin{definition}\lbl{defn-lm-intro}
 Given a finite subset $A\subset\N$, we put 
 \[ K(A) = \bigWedge_{a\in A} K(a) \in\CB, \]
 and we let $\lm_A\:\CB\to\CB$ denote the Bousfield localisation
 functor with respect to $K(A)$.  Fix $n^*\geq 1$, and put
 $N=\{0,\dotsc,n^*-1\}$.  Let $\Lm$ denote the monoid of endofunctors of
 $\CB$ generated by all the $\lm_A$ for $A\sse N$.
\end{definition}

Our original goal was to describe the structure of $\Lm$.  It turns
out to be natural to consider instead a certain monoid $\QQ$ that acts
on $\CB$ and includes all the functors $\lm_A$.  This differs from
$\Lm$ in that~(a) we do not know whether the map
$\QQ\to\pi_0\End(\CB)$ is injective, and~(b) the image of $\QQ$ is
strictly larger than $\Lm$.

We will take an axiomatic approach, which will also cover some
non-chromatic examples.  However, in this introduction we will focus
on the chromatic case.  As an example, consider the functor
\[ F = \lm_{013}\lm_{023} =
      L_{K(0)\Wedge K(1)\Wedge K(3)}L_{K(0)\Wedge K(2)\Wedge K(3)},
\]
and the submonoid $\bc{F}\leq\Lm$ that it generates.  We do not know a
very easy way to see that $|\bc{F}|$ is even finite, but we will show
that in fact $|\bc{F}|=3$.

In order to motivate our general approach, we recall the theory of
chromatic fracture squares.  As a special case of a fact that was
already known to Bousfield, for all $m<n$ and $X\in\CB$, there is a
homotopy cartesian square
\begin{center}
 \begin{tikzcd}
   L_{K(m)\Wedge K(n)}X \arrow[r] \arrow[d] &
   L_{K(n)}X \arrow[d] \\
   L_{K(m)}X \arrow[r] &
   L_{K(m)}L_{K(n)}X,
 \end{tikzcd}
\end{center}
which exhibits $L_{K(m)\Wedge K(n)}X$ as the homotopy inverse limit of
a certain subdiagram of $\psi(X)$.  By the same methods, if $|A|=d$
then one can exhibit a homotopy cartesian hypercube of dimension $d$,
which expresses $\lm_AX$ as a homotopy inverse limit of terms of the
form $L_{K(t_1)}\dotsb L_{K(t_r)}X$.  We will again call this
phenomenon \emph{chromatic fracture}.

We now give an initial version of our main definitions and results.
To make them precise, we will need a significant amount of
foundational work, as will be discussed below.

\begin{definition}\lbl{defn-P}
 Let $\PP$ be the set of subsets of $N$, ordered by inclusion.  For
 $A,B\in\PP$, we write $A\angle B$ if $a\leq b$ for all $a\in A$ and
 $b\in B$.  If $A=\{a_1,\dotsc,a_r\}$ with $a_1<\dotsb<a_r$, we put
 \[ \phi_A(X) = L_{K(a_1)} \dotsb L_{K(a_r)} X. \]
\end{definition}
\begin{remark}\lbl{rem-angle}
 Note that $A\angle B$ is vacuously satisfied if $A=\emptyset$ or
 $B=\emptyset$, and because of this, the relation is not transitive.
 (It is not reflexive or symmetric either.) 
\end{remark}

\begin{definition}\lbl{defn-Q}
 Let $\QQ$ be the set of all subsets of $\PP$ that are closed upwards,
 ordered by reverse inclusion.  We define $u\:\PP\to\QQ$ by
 $uA=\{B\st A\sse B\}$, so $u$ is a morphism of posets.  We also
 define $v\:\PP\to\QQ$ by $vA=\{B\st B\cap A\neq\emptyset\}$, so $v$
 is order-reversing.
\end{definition}

\begin{remark}\lbl{rem-PQ-ends}
 In $\PP$, the smallest element is $\emptyset$ and the largest element
 is $N$.  In $\QQ$, the smallest element is $u\emptyset=\PP$ and the
 largest element is $\emptyset$.  The element $uN$ is second-largest
 in $\QQ$, in the sense that every element $U\in\QQ$ with
 $U\neq\emptyset$ satisfies $U\leq uN$.  
\end{remark}

\begin{lemma}\lbl{lem-mu}
 There is a map $\mu\:\QQ\tm\QQ\to\QQ$ of posets given by 
 \[ \mu(U,V) = U * V =
     \{A\cup B \st A\in U,\; B\in V,\; A\angle B\}.
 \]
This operation is associative, with
 \[ U * V * W = \{A\cup B\cup C\st A\in U,\; B\in V,\; C\in W,\;
                  A\angle B,\; A\angle C,\; B\angle C\},
 \]
 and $u\emptyset$ is a two-sided identity element.
 Moreover it is distributive on both sides with respect to the union.
\end{lemma}
\begin{proof}
 Suppose that $A\in U$, $B\in V$, $A\angle B$ and $A\cup B\sse C$.
 We can then choose $t$ such that $a\leq t$ for all $a\in A$, and
 $t\leq b$ for all $b\in B$.  We put $A'=\{c\in C\st c\leq t\}$ and
 $B'=\{c\in C\st t\leq c\}$.  Then $A\sse A'$ so $A'\in U$, and
 $B\sse B'$ so $B'\in V$.  We also have $C=A'\cup B'$ with
 $A'\angle B'$, so $C\in U*V$.  This proves that $U*V$ is closed
 upwards, so we have indeed defined a map $\mu\:\QQ\tm\QQ\to\QQ$.  It
 is clear that if $U\sse U'$ and $V\sse V'$ then $U*V\sse U'*V'$, so
 $\mu$ is a morphism of posets.  All remaining claims are also easy.
\end{proof}

\begin{remark}\lbl{rem-kp}
 We can define $\kp\:\QQ\to\PP$ by $\kp(U)=\{n\st\{n\}\in U\}$.  This
 is order-reversing and satisfies $\kp(U*V)=\kp(U)\cap\kp(V)$ and
 $\kp(vA)=A$ and
 \[ \kp(uA) = \begin{cases}
      N & \text{ if } A = \emptyset \\
      A & \text{ if } |A| = 1 \\
      \emptyset & \text{ if } |A| > 1.
    \end{cases}
 \]
 However, $\kp$ is very far from being injective, so this gives only
 crude insight into the monoid structure of $\QQ$.  We do not know any
 better example of a homomorphism to a more familiar monoid.
\end{remark}

It is familiar that we can regard posets as categories with hom sets
of size at most one.  The above lemma then makes $\QQ$ into a monoidal
category.  We also have a monoidal category $\End(\CB)$ of
endofunctors of $\CB$, with composition as the monoidal product.

We now give a preliminary statement of our main result.  This will be
imprecise in various ways, to be discussed below; the imprecision will
be removed in the main body of the paper.

\begin{theorem}\lbl{thm-main-intro}
 There is a strong monoidal functor $U\mapsto\tht_U$ from $\QQ$ to
 $\End(\CB)$, with $\tht_{uA}=\phi_A$ and $\tht_{vA}=\lm_A$.  We
 therefore have $\tht_U\tht_V(X)\simeq\tht_{U*V}(X)$, and there are
 compatible natural maps $\tht_U(X)\to\tht_V(X)$ whenever $U\leq V$ in
 $\QQ$, or equivalently $U\supseteq V$.  The definition is that
 $\tht_U(X)$ is the homotopy inverse limit of the objects $\phi_A(X)$
 for $A\in U$.
\end{theorem}

As $\QQ$ is finite and the image of $\tht$ contains the generators of
$\Lm$, we see in particular that $\Lm$ is finite.  We do not know
whether $\tht$ is injective.  We can spell out the relationship
between $\QQ$ and $\Lm$ a little more explicitly as follows:

\begin{definition}\lbl{defn-thread}
 Let $\LA=(A_1,\dotsc,A_r)$ be a list of subsets of $N$.  A
 \emph{thread} for $\LA$ is a list $(a_1,\dotsc,a_r)$ with
 $a_i\in A_i$ for all $i$ and $a_1\leq a_2\leq\dotsb\leq a_r$.  A
 \emph{thread set} for $\LA$ is a subset $A^*\sse N$ such that there
 exists a thread contained in $A^*$.  We write $T(\LA)$ for the set of
 all thread sets.  This is clearly closed upwards, so $T(\LA)\in\QQ$.
 We also write $\lm_{\LA}$ for the composite
 $\lm_{A_1}\dotsb\lm_{A_r}$. 
\end{definition}

\begin{proposition}\lbl{prop-thread}
 In $\QQ$ we have $T(\LA)=vA_1*\dotsb *vA_r$.  Thus,
 Theorem~\ref{thm-main-intro} implies that $\lm_{\LA}=\tht_{T(\LA)}$. 
\end{proposition}
\begin{proof}
 This follows from the definitions by a straightforward induction.
\end{proof}

\begin{example}\lbl{eg-thread}
 We previously mentioned the functor $F=\lm_{013}\lm_{023}$.  This is
 $\tht_U$, where
 \begin{align*}
  U &= v\{0,1,3\} * v\{0,2,3\} = T(\{0,1,3\},\{0,2,3\}) \\
  &= \{A\st 0\in A \text{ or } 3 \in A \text{ or } \{1,2\} \sse A\}.
 \end{align*}
 A check of cases shows that $U*U=v\{0,3\}$ and then that
 $U^{*k}=v\{0,3\}$ for $k\geq 2$.  Thus, the monoid generated by $F$
 is $\{1,F,\lm_{03}\}$.
\end{example}

\begin{example}\lbl{eg-not-thread}
 Take $N=\{0,1,2\}$ and $U=\{\{0,1\},\{1,2\},\{0,1,2\}\}\in\QQ$.  We
 claim that $U\neq T(\LA)$ for any $\LA$, so $\QQ$ really is different
 from $\Lm$.  Indeed, suppose that $U=T(\LA)$ with
 $\LA=(A_1,\dotsc,A_r)$.  Then $\{1\}\not\in T(\LA)$, so we can choose
 $m$ with $1\not\in A_m$.  On the other hand, we have
 $\{0,1\}\in T(\LA)$, which means that there exists $p$ with
 $1<p\leq r$ and $0\in A_i$ for $i<p$ and $1\in A_i$ for $i\geq p$.
 From this it is clear that $p>m$.  Similarly, as $\{1,2\}\in T(A)$
 there must exist $q$ with $1\leq q<r$ and $1\in A_i$ for $i\leq q$
 and $2\in A_i$ for $i>q$.  From this we see that $q<m$ and so
 $q\leq p-2$.  It follows in turn that $\{0,2\}\in T(\LA)$,
 contradicting our assumption that $T(\LA)=U$.
\end{example}

One problem with our preliminary statement is that the formation of
homotopy limits requires diagrams that commute in some underlying
model category, whereas localisation functors are merely characterised
by a homotopical universal property.  To avoid these issues, we will
need some foundational work with derivators and anafunctors.  For any
$X$, we would like to construct a coherent diagram containing all of
the objects $\tht_UX$ and all the natural morphisms between them.
Instead of constructing this directly, we will define a category of
potential candidates (Definition~\ref{defn-fracture-obj}), and prove
that an appropriate forgetful functor to $\CB$ is an equivalence
(Proposition~\ref{prop-fracture-obj}).  To make this work smoothly, we
need a version that works uniformly when $X$ is not just a single
spectrum, but is itself a coherent diagram.  This is precisely the
kind of issue for which the theory of derivators is designed.
However, it will still work out that $\tht_U$ is not really an honest
functor, but is instead a kind of fraction in which we formally invert
an equivalence of derivators.  We will also need some foundational
work to support this.

\begin{remark}\lbl{rem-lean}
 This paper contains a number of results about the combinatorial
 homotopy theory of $\PP$, $\QQ$ and various other posets constructed
 from these.  All of these results have been formalised in the Lean
 proof assistant.  A snapshot of the code will be deposited on the
 arXiv.  Active development will continue at
 \url{https://github.com/NeilStrickland}.   
\end{remark}

\section{Basic definitions}
\label{sec-basic}

\begin{definition}\lbl{defn-context}
 We will fix a compactly generated triangulated category $\CB$.  We
 also fix an integer $n^*\geq 1$ and put $N=\{0,\dotsc,n^*-1\}$ as
 before.  We then fix a family of homology theories
 $K(n)_*\:\CB\to\Ab_*$ for $n\in N$ and put
 $K(A)_*=\bigoplus_{a\in A}K(a)_*$ for any $A\sse N$.  We let $\lm_A$
 denote the localisation with respect to the localizing subcategory of
 $K(A)_*$- acyclics.  We assume the following condition, which we call
 the \emph{fracture axiom}: if $A$ is a nonempty subset of $N$, and
 $b\in N$ with $b>\max(A)$, then $K(b)_*\lm_A(X)=0$ for all $X$.
\end{definition}

\begin{remark}
 Since later we will employ the theory of stable derivators we need to
 assume that our triangulated category $\CB$ is the underlying
 category of a stable derivator, i.e. $\CB \simeq \CC(e)$ for some
 derivator $\mathcal{C}$. This condition is not too restrictive and it
 is verified if $\CB$ has a geometric model (see \cite[Theorem
 6.11]{ci:idc} or the easier result \cite[Proposition
 1.36]{groth:derpointstable} for combinatorial model categories).
\end{remark}

\begin{lemma}\lbl{lem-fracture-axiom}
 The fracture axiom implies the following statement
 (which we call the \emph{extended fracture axiom}): if $A,B\sse N$
 with $A\angle B$ and $K(B)_*(X)=0$, then $K(B)_*(\lm_A(X))=0$.
\end{lemma}
\begin{proof}
 If $A=\emptyset$ then $K(A)_*=0$ and so $\lm_A=0$ and everything is
 trivial.  We can thus assume that $A\neq\emptyset$, so $\max(A)$ is
 defined.  The assumption $A\angle B$ then means that $b\geq\max(A)$
 for all $b\in B$.  We are given that $K(B)_*(X)=0$, or in other words
 that $K(b)_*(X)=0$ for all $b\in B$.  We want to prove that
 $K(b)_*(\lm_A(X))=0$.  If $b>\max(A)$ then this is immediate from the
 fracture axiom.  This just leaves the case where $b=\max(A)$, so
 $b\in A$.  The map $X\to\lm_A(X)$ is a $K(A)$-equivalence, so it is a
 $K(b)$-equivalence (because $b\in A$), and $K(b)_*(X)=0$ by
 assumption, so $K(b)_*(\lm_A(X))=0$ as required. 
\end{proof}

All our examples will be verified using the following result:

\begin{proposition}\lbl{prop-eg-generic}
 Let $\CB$ be a stable homotopy category as in~\cite{hopast:ash} (so
 it has a closed symmetric monoidal structure compatible with the
 triangulation).  Suppose we have $N$ as before and objects
 $K(n)\in\CB$ representing the homology theories
 $K(n)_*(X)=\pi_*(K(n)\Smash X)$.  Suppose we also have objects
 $T(n)\in\CB$, and that the following axioms are satisfied:
 \begin{itemize}
  \item[(a)] $T(n)$ is dualizable for all $n$.
  \item[(b)] For $m<n$ we have $K(m)\Smash T(n)=0$.
  \item[(c)] For any object $X$ and any $n$ we have $K(n)_*(X)=0$ iff
   $K(n)\Smash X=0$ iff $K(n)\Smash T(n)\Smash X=0$.
 \end{itemize}
 Then the fracture axiom is satisfied. 
\end{proposition}
\begin{proof}
 Suppose that $b>\max(A)$.  We need to show that
 $K(b)\Smash \lm_A(X)=0$, and by axiom~(c) it will suffice to show
 that $K(b)\Smash T(b)\Smash\lm_A(X)=0$.  For this it will suffice to
 show that $T(b)\Smash\lm_A(X)=0$, or that the identity map of
 $T(b)\Smash\lm_A(X)$ is zero, or that the adjoint map
 \[ DT(b) \Smash T(b) \Smash \lm_A(X) \to \lm_A(X) \]
 is zero.  Here $K(A)\Smash T(b)=0$ by axiom~(b), so the source of the
 above map is $K(A)$-acyclic, whereas the target is $K(A)$-local; this
 implies that the map is zero as required.
\end{proof}

\begin{example}\lbl{eg-derived}
 For the simplest example, let $\CB$ be the derived category of
 modules over $\Zpl$, and put
 \begin{align*}
  K(0) &= \Q   & K(1) &= \Z/p \\
  T(0) &= \Zpl & T(1) &= \Z/p.
 \end{align*}
 It is then straightforward to check the hypotheses of
 Proposition~\ref{prop-eg-generic}, so the fracture axiom is
 satisfied.
\end{example}

\begin{example}\lbl{eg-chromatic}
 For the motivating example, fix a prime $p$.  Let $\CB_0$ denote the
 category of symmetric spectra of simplicial sets, equipped with the
 $p$-localisation of the usual model structure.  Put $\CB=\Ho(\CB_0)$
 (so this is the usual stable homotopy category of $p$-local spectra).
 For any $n\in N=\{0,\dotsc,n^*-1\}$, let $K(n)$ denote the Morava
 $K$-theory spectrum of height $n$ at the prime $p$, and let $T(n)$ be
 any finite $p$-local spectrum of type $n$.  It is again
 straightforward to check the hypotheses of
 Proposition~\ref{prop-eg-generic}, so the fracture axiom is
 satisfied. 
\end{example}

\begin{example}\lbl{eg-equivariant}
 For another example, fix a cyclic group $G$ of order $p^d$ for some
 prime $p$ and $d\geq 0$.  Put $N=\{0,\dotsc,d\}$.  For $n\in N$ let
 $H_n$ be the unique subgroup of order $p^{d-n}$ in $G$, and put
 $Q_n=G/H_n$, so that $|Q_n|=p^n$.  In the $G$-equivariant stable
 homotopy category, put $T(n)=(Q_n)_+$ and
 \[ K(n) = (Q_n)_+\Smash \tSg EQ_{n+1}, \]
 where $\tSg$ denotes the unreduced suspension.  For the case $n=d$,
 this should be interpreted as $K(d)=(Q_d)_+=G_+$.  We find that the
 geometric fixed points are
 \[ \phi^{H_m} T(n) =
    \begin{cases}
     0 & \text{ if } m < n \\
     (Q_n)_+ & \text{ if } m \geq n,
    \end{cases} \hspace{4em}
    \phi^{H_m} K(n) =
    \begin{cases}
     0 & \text{ if } k \neq n \\
     S^0 & \text{ if } k = n.
    \end{cases}
 \]
 Recall that $\phi^{H_m}$ preserves smash products, and that $X=0$ iff
 $\phi^{H_m}(X)=0$ (in the nonequivariant stable category) for all
 $m$.  Using this, it is not hard to check the hypotheses of
 Proposition~\ref{prop-eg-generic}, and we again see that the fracture
 axiom is satisfied.
\end{example}

\begin{remark}\lbl{rem-general-groups}
 It would of course be interesting to see what one could say about
 more general finite groups, where the subgroup lattice is more
 complicated.  We suspect that this will be substantially harder; we
 may return to the question in future work.
\end{remark}

\section{Derivators and homotopy (co)limits}
\label{sec-derivators}

We will use the theory of stable derivators, mostly
following~\cite{groth:derpointstable}.  

\begin{definition}\lbl{defn-poset-cat}
 Let $\POSet$ denote the strict $2$-category of finite posets, so the
 $0$-cells are finite posets, and the $1$-cells are nondecreasing
 maps.  Given two non-decreasing maps $u,v\:P\to Q$, there is one
 $2$-cell from $u$ to $v$ if $u(p)\leq v(p)$ for all $p$, and no
 $2$-cells otherwise.  We write $[n]$ for the set $\{0,\dotsc,n\}$
 with its usual order, so $[n]\in\POSet$.  Note that $[0]$ is the
 terminal poset, which will also be denoted by $e$.
\end{definition}

\begin{definition}\lbl{defn-prederivator}
 For us, a \emph{prederivator} is a strict $2$-functor
 $\CC\:\POSet^{\opp}\to\CAT$.  More explicitly, it consists of
 \begin{itemize}
  \item[(a)] For every finite poset $P$, a category $\CC(P)$.
  \item[(b)] For every morphism $u\:P\to Q$, a functor
   $u^*\:\CC(Q)\to\CC(P)$, such that $1^*=1$ and $(v\circ u)^*=u^*v^*$
   on the nose.
  \item[(c)] For every inequality $u\leq v$ between morphisms $P\to Q$,
   a natural map $u^*\to v^*$, satisfying some evident axioms.
 \end{itemize}
\end{definition}
(By restricting attention to finite posets rather than more general
categories, we are following~\cite{groth:derpointstable}*{Remark 1.8}.)
\begin{remark}\lbl{rem-underlying}
 For any prederivator $\CC$ we have a category $\CC(e)$, which we call
 the \emph{underlying category} of $\CC$.  We will often think of this
 as the key ingredient, with the other categories just adding extra
 structure to $\CC(e)$ in some sense.
\end{remark}

\begin{definition}\lbl{defn-stable-derivator}
 A \emph{derivator} is a prederivator in which all the functors $u^*$
 have left and right adjoints with certain compatibility properties,
 as specified in~\cite{groth:derpointstable}*{Definition 1.10}.  These
 adjoints can be thought of as homotopy right and left Kan extensions.
 A \emph{stable derivator} is a derivator $\CC$ subject to some
 additional conditions:
 \begin{itemize}
  \item[(a)] $\CC$ should be \emph{strong}.  To explain this, note that
   for any $P$ there are evident inclusions $i_0,i_1\:P\to[1]\tm P$
   with $i_0\leq i_1$.  We therefore have functors
   $i_0^*,i_1^*\:\CC([1]\tm P)\to\CC(P)$, together with a natural map
   between them.  These can be combined in an obvious way to get a
   functor $\CC([1]\tm P)\to\CC(P)^{[1]}$, and the strongness condition
   is that this functor should be full and essentially surjective (for
   all $P$).  This is~\cite{groth:derpointstable}*{Definition 1.13}.
  \item[(b)] $\CC$ should be \emph{pointed}, which means that $\CC(e)$
   should have an object that is both initial and terminal.  Many
   consequences of this condition are investigated
   in~\cite{groth:derpointstable}*{Section 3}.
  \item[(c)] Homotopy pushouts squares in $\CC$ must coincide with
   homotopy pullback squares, in the sense spelled out
   in~\cite{groth:derpointstable}*{Definition 4.1}.
 \end{itemize}
\end{definition}

\begin{remark}\lbl{rem-derivator-shift}
 If $\CC$ is a derivator and $P$ is a finite poset, then we can define
 a new derivator $\CC^P$ by $\CC^P(Q)=\CC(P\tm Q)$.  This is called a
 \emph{shifted derivator} and it is often a useful device.
 In~\cite{groth:derpointstable}, Theorem~1.31 proves that $\CC^P$ is
 indeed a derivator, and Proposition~2.6 proves that for any
 $u\:P\to Q$, the resulting morphism $u^*\:\CC^Q\to\CC^P$ preserves
 left and right homotopy Kan extensions.
\end{remark}

\begin{remark}\lbl{rem-dual-derivator}
 For every derivator $\CC$ there is a dual derivator
 $\CC^\opp(P)=\CC(P^\opp)^\opp$, as
 in~\cite{groth:derpointstable}*{Definition 1.15} and surrounding
 discussion. 
\end{remark}

\begin{definition}\lbl{defn-hocolim}
 Let $\CC$ be a derivator.  Let $P$ be a finite poset, and let $c$ be
 the unique morphism $P\to e$, so we have functors
 $c_!,c_*\:\CC(P)\to\CC(e)$.  We define $\hocolim_PX=c_!(X)$ and
 $\holim_P(X)=c_*(X)$.
\end{definition}

\begin{definition}\lbl{defn-comma-posets}
 Let $f\:P\to Q$ be a morphism of finite posets, and suppose that
 $q\in Q$.  We use the following notation for comma posets:
 \begin{align*}
  f/q &= \{p\in P\st f(p)\leq q\} \\
  q/f &= \{p\in P\st q \leq f(p)\}.
 \end{align*}
\end{definition}

\begin{remark}\lbl{rem-kan}
 We can now recall the key axiom
 from~\cite{groth:derpointstable}*{Definition 1.10}, which is known as
 the Kan formula.  Suppose we have a morphism $f\:P\to Q$, a derivator
 $\CC$, an object $X\in\CC(P)$, and an element $q\in Q$.  The element
 $q$ gives $i_q\:e\to Q$, and we want to understand the object
 $i_q^*f_*X\in\CC(e)$.  There is an evident inclusion $j\:q/f\to P$ so
 we have $j^*X\in\CC(q/f)$ and $\holim_{q/f}j^*X\in\CC(e)$.   Using
 various adjunctions one can write down a natural map
 $i_q^*f_*(X)\to\holim_{q/f}j^*X$, and the axiom says that this should
 be an isomorphism.  Dually, the natural map
 $\hocolim_{f/q}j^*X\to i_q^*f_!(X)$ should also be an isomorphism.
 
 This is a direct generalization of the Kan formula for Kan extensions
 of functors in the usual categorical setting. That is, for a type of
 derivators called \textit{represented}, the above isomorphisms are
 exactly the expression of right or left Kan extenions via limits and
 colimits respectively. See the discussion following~\cite[Definition
 1.9]{groth:derpointstable}.
\end{remark}

\begin{theorem}\lbl{thm-triangulation}
 Let $\CC$ be a stable derivator.  Then each category $\CC(P)$ has a
 natural structure as a triangulated category.  Moreover, for each
 $u\:P\to Q$, the corresponding functor $u^*\:\CC(Q)\to\CC(P)$ has a
 canonical natural isomorphism $u^*\Sg\to\Sg u^*$ with respect to
 which it is an exact functor, and the same applies to the left and
 right adjoint functors $u_!,u_*\:\CC(P)\to\CC(Q)$.
\end{theorem}
\begin{proof}
 This is~\cite{groth:derpointstable}*{Corollary 4.19}.
\end{proof}

\begin{definition}\lbl{defn-support}
 Let $\CC$ be a stable derivator, and let $P$ be a finite poset.  Each
 $p\in P$ gives a morphism $i_p\:e\to P$ and thus a functor
 $i_p^*\:\CC(P)\to\CC(e)$.  Given $X\in\CC(P)$, we write $X_p$ for
 $i_p^*(X)\in\CC(e)$.  We also put $\supp(X)=\{p\st X_p\neq 0\}\sse P$,
 and call this the \emph{support} of $X$.  Given $Q\sse P$, we say
 that $X$ is \emph{supported in $Q$} if $\supp(X)\sse Q$.  We write
 $\CC_Q(P)$ for the full subcategory of $\CC(P)$ consisting of objects
 supported in $Q$.
\end{definition}

\begin{remark}\lbl{rem-supp}
 Derivators are defined
 in~\cite{groth:derpointstable}*{Definition 1.10}, and the second
 axiom says that a morphism $f\:X\to Y$ in $\CC(P)$ is an isomorphism
 iff $f_p\:X_p\to Y_p$ is an isomorphism for all $p$.  From this we
 see that $\supp(X)=\emptyset$ iff $X=0$.  Similarly, for any
 $M\sse P$ we have an inclusion $j\:M\to P$ and we write $X|_M$
 for $j^*(Q)$.  We then find that $X$ is supported in $Q$ iff
 $X|_{P\sm Q}=0$. 
\end{remark}

\begin{definition}\lbl{defn-sieve}
 Let $P$ be a finite poset, and let $Q$ be a subset of $P$.
 \begin{itemize}
  \item[(a)] We say that $Q$ is a \emph{sieve} if it is closed
   downwards, so that whenever $p\leq q$ with $q\in Q$ we also have
   $p\in Q$.
  \item[(b)] We say that $Q$ is a \emph{cosieve} if it is closed
   upwards, so that whenever $q\leq p$ with $q\in Q$ we also have
   $p\in Q$.
 \end{itemize}
 If $Q\sse P$ is a (co)sieve, we also say that the inclusion morphism
 $Q\to P$ is a (co)sieve.  More generally, if $i\:Q\to P$ is an
 embedding (so $i(q)\leq i(q')$ iff $q\leq q'$) and $i(Q)$ is a
 (co)sieve, we say that $i$ is a (co)sieve.  This is the finite poset
 version of~\cite{groth:derpointstable}*{Definition 1.28}.
\end{definition}

\begin{lemma}\lbl{lem-embedding}
 If $\CC$ is a derivator and $j\:Q\to P$ is an embedding then the
 counit map $j^*j_*X\to X$ and the unit map $X\to j^*j_!X$ are both
 isomorphisms (for all $X\in\CC(Q)$).  Thus, the functors $j_*$ and
 $j_!$ are both full and faithful embeddings.  In particular, this
 holds if $j$ is a sieve or a cosieve.
\end{lemma}
\begin{proof}
 We first consider the counit map $j^*j_*X\to X$.  By the derivator
 axiom Der2, it will suffice to prove that the induced map
 $(j_*X)_{j(q)}=(j^*j_*X)_q\to X_q$ is an isomorphism in $\CC(e)$ for
 all $q\in Q$.  Put $j(q)/j=\{a\in Q\st j(q)\leq j(a)\}$ and let
 $k\:j(q)/j\to Q$ be the inclusion.  The Kan formula identifies
 $(j_*X)_{j(q)}$ with the homotopy limit of $k^*(X)\in\CC(j(q)/q)$.
 Note that $q$ is initial in $j(q)/j$, so the inclusion
 $i_q\:e\to j(q)/j$ is left adjoint to $c\:j(q)/j\to e$, so the
 homotopy limit functor $c_*$ is the same as $i_q^*$
 by~\cite{groth:derpointstable}*{Lemma 1.23}.  This gives
 $(j_*X)_{j(q)}=c_*k^*(X)=(ki_q)^*(X)=X_q$ as required.  This proves
 that the counit map $j^*j_*X\to X$ is an isomorphism.  To prove that
 the unit map $X\to j^*j_!X$ is also an isomorphism, we can either
 give a similar argument, or take adjoints, or appeal to a kind of
 self-duality of the theory of derivators.  From the isomorphism
 $X\simeq j^*j_!X$ we obtain $[W,X]\simeq[W,j^*j_!X]=[j_!W,j_!X]$, so
 $j_!$ is full and faithful.  A similar argument shows that $j_*$ is
 full and faithful.
\end{proof}

\begin{proposition}\lbl{prop-sieve}
 Let $\CC$ be a stable derivator, and let $j\:Q\to P$ be a sieve.
 Then:
 \begin{itemize}
  \item[(a)] The functor $j_*\:\CC(Q)\to\CC(P)$ has a right adjoint
   denoted by $j^!$, as well as the left adjoint $j^*$ that exists by
   the definition of $j_*$.
  \item[(b)] The unit map $X\to j^!j_*(X)$ is an isomorphism (as are
   the unit map $X\to j^*j_!X$ and the counit map $j^*j_*X\to X$, as
   we saw in Lemma~\ref{lem-embedding}).  
  \item[(c)] The functor $j_*$ gives an equivalence from $\CC(Q)$ to
   $\CC_Q(P)$, with inverse given by $j^*$ or $j^!$.
  \item[(d)] If $Y\in\CC(Q)$ corresponds to $X\in\CC_Q(P)$ then
   $\holim_Q(Y)\simeq\holim_P(X)$. 
 \end{itemize}
\end{proposition}
\begin{proof}
 Most of parts~(a) to~(c) can be obtained by combining Definition~3.4,
 Proposition~3.6 and Corollary~3.8 from~\cite{groth:derpointstable}.
 More specifically, 3.6(ii) says that $j_*$ is full and faithful with
 $\CC_Q(P)$ as the essential image, and 3.8 gives us the functor
 $j^!$.  From Lemma~\ref{lem-embedding} we have $1\simeq j^*j_*$, and
 by taking right adjoints we get $1\simeq j^!j_*$.  We leave it to the
 reader to check that this natural isomorphism is just the unit
 map.   Given this, the claim~(d) just reduces to
 $(c_P)_*j_*\simeq(c_Q)_*$, where $c_P$ and $c_Q$ are the unique
 morphisms $P\to e$ and $Q\to e$.  But this is clear because
 $c_Pj=c_Q$. 
\end{proof}

\begin{proposition}\lbl{prop-cosieve}
 Let $\CC$ be a stable derivator, and let $i\:R\to P$ be a cosieve.
 Then:
 \begin{itemize}
  \item[(a)] The functor $i_!\:\CC(R)\to\CC(P)$ has a left adjoint
   denoted by $i^?$, as well as the right adjoint $i^*$ that exists by
   the definition of $i_!$.
  \item[(b)] The counit map $i^?i_!(X)\to X$ is an isomorphism (as
   are the unit map $X\to i^*i_!X$ and the counit map $i^*i_*X\to X$, as
   we saw in Lemma~\ref{lem-embedding}).  
  \item[(c)] The essential image of $i_!$ is precisely $\CC_R(P)$,
   so in fact we have an equivalence $\CC(R)\simeq\CC_R(P)$.
  \item[(d)] If $Z\in\CC(R)$ corresponds to $X\in\CC_R(P)$ then
   $\hocolim_R(Z)\simeq\hocolim_P(X)$. 
 \end{itemize}
\end{proposition}
\begin{proof}
 This is dual to the previous proposition.
\end{proof}

\begin{proposition}\lbl{prop-recollement}
 We now want to combine the above two propositions.  Suppose that
 $Q\sse P$ is a sieve, and let $R=P\sm Q$ be the complementary
 cosieve.  Let $Q\xra{j}P\xla{i}R$ be the inclusions, so we have
 functors as follows, with each functor left adjoint to the one below
 it. 
 \begin{center}
  \begin{tikzcd}[column sep=large]
    \CC(Q)
     \arrow[rr,bend left =30,"j_!" ]
     \arrow[rr,bend right=10,"j_*"'] &&
    \CC(P)
     \arrow[ll,bend right=10,"j^*"']
     \arrow[ll,bend left =30,"j^!" ]
     \arrow[rr,bend left =30,"i^?" ]
     \arrow[rr,bend right=10,"i^*"'] &&
    \CC(R)
     \arrow[ll,bend right=10,"i_!"']
     \arrow[ll,bend left =30,"i_*" ]
  \end{tikzcd}
 \end{center}
 \begin{itemize}
  \item[(a)] The composites $i^?j_!$, $j^*i_!$, $i^*j_*$ and $j^!i_*$
   (obtained by composing two functors at the same level in the
   diagram) are all zero. (We have nothing systematic to say about any
   other composite functors between $\CC(Q)$ and $\CC(R)$.)
  \item[(b)] The six adjunctions in the diagram involve six (co)unit
   maps to or from the identity functor of $\CC(P)$.  These fit into a
   diagram as follows, in which every straight line is part of a
   natural distinguished triangle:
   \begin{center}
    \begin{tikzpicture}[scale=2]
     \draw (0,0) node{$1$};
     \draw (  0:1) node{$j_*j^*$};
     \draw ( 60:1) node{$j_*j^!$};
     \draw (120:1) node{$i_!i^?$};
     \draw (180:1) node{$i_!i^*$};
     \draw (240:1) node{$i_*i^*$};
     \draw (300:1) node{$j_!j^*$};
     \draw[->] (  0:0.1) -- (  0:0.85);
     \draw[<-] ( 60:0.1) -- ( 60:0.85);
     \draw[->] (120:0.1) -- (120:0.85);
     \draw[<-] (180:0.1) -- (180:0.85);
     \draw[->] (240:0.1) -- (240:0.85);
     \draw[<-] (300:0.1) -- (300:0.85);
    \end{tikzpicture}
   \end{center}
 \end{itemize}
\end{proposition}
\begin{proof}
 This is just a reorganisation of information extracted
 from~\cite{groth:derpointstable}*{Example 4.25}.  We will explain
 some parts of the argument.  The Kan formula expresses
 the object $(j^*i_!Z)_q=(i_!Z)_{j(q)}$ as a homotopy colimit over a
 certain comma poset, but the (co)sieve properties of $i$ and $j$
 ensure that this comma poset is empty, so $j^*i_!=0$.  By taking
 left and right adjoints repeatedly we deduce that the other
 composites in~(a) are also zero.  The horizontal composite
 $f\:i_!i^*\to j_*j^*$ is adjoint to a map $i^*\to i^*j_*j^*$, and
 $i^*j_*=0$, so it follows that $f=0$.  The same kind of argument
 shows that the two other straight line composites are also zero.  For
 the remaining points we refer to~\cite{groth:derpointstable}, and to
 the paper~\cite{he:trm} that is cited there.
\end{proof}

\section{Homotopy theory of partially ordered sets}
\label{sec-homotopy}

\begin{definition}\lbl{defn-cart-cl}
 Let $\POSet(P,Q)$ be the set of monotone maps from $P$ to $Q$.  We
 give this the partial order such that $f\leq g$ iff $f(p)\leq g(p)$
 in $Q$ for all $p\in P$.  It is easy to see that this makes the
 category of finite posets into a cartesian closed category.
\end{definition}

\begin{definition}\lbl{defn-geom}
 Let $P$ be a finite poset.  Recall that a subset $\sg\sse P$ is a
 \emph{chain} if the induced order on $\sg$ is total.  Given a map
 $x\:P\to[0,1]$, we define $\supp(x)=\{p\st x(p)>0\}$.  We define
 $|P|$ to be the set of maps $x\:P\to[0,1]$ such that $\supp(x)$ is a
 chain and $\sum_{p\in P}x(p)=1$.  We call this the \emph{geometric
  realisation} of $P$.  A standard argument shows that this gives a
 functor $\POSet\to\Top$ that preserves finite coproducts and finite
 limits.  In particular, we can apply geometric realisation to the
 evaluation map $\POSet(P,Q)\tm P\to Q$, and then take adjoints, to
 get a continuous map
 \[ |\POSet(P,Q)| \to \Top(|P|,|Q|). \]
\end{definition}

\begin{definition}\lbl{defn-pi-zero}
 For any finite poset $P$, we define $\pi_0(P)$ to be the quotient of
 $P$ by the smallest equivalence relation such that $p\sim q$ whenever
 $p\leq q$.  This is easily seen to be the same as the set of path
 components of $|P|$.  It gives a functor from finite posets to finite
 sets, which preserves finite products and coproducts.  It follows
 formally that we can construct a quotient category of $\POSet$ with
 morphism sets $\pi_0(\POSet(P,Q))$.  We call this the \emph{strong
  homotopy category} of finite posets.  It also follows that if $f$
 and $g$ lie in the same equivalence class of $\pi_0(\POSet(P,Q))$
 then the resulting maps $|f|$ and $|g|$ are homotopic (by a
 straight-line homotopy, in the basic case where $f\leq g$ or
 $g\leq f$).  Thus, geometric realisation gives a functor from the
 strong homotopy category of posets to the homotopy category of
 topological spaces.
\end{definition}

\begin{remark}\lbl{rem-adjoint-strong}
 Suppose we have morphisms $f\:P\to Q$ and $g\:Q\to P$ that are
 adjoint, in the sense that $f(p)\leq q$ iff $p\leq g(q)$.  We then
 have (co)unit inequalities $1\leq gf$ and $fg\leq 1$, showing that
 $fg$ and $gf$ give identities in the strong homotopy category, and
 thus that $f$ and $g$ are strong homotopy equivalences.
\end{remark}

\begin{definition}\lbl{defn-strongly-contractible}
 We say that $P$ is \emph{strongly contractible} if the map
 $c_P\:P\to e$ is a strong homotopy equivalence.  
\end{definition}

We note that this holds if $P$ has a smallest element or a largest
element.  We also note that if $P$ is a strongly contractible poset,
then $|P|$ is a contractible space.

\begin{definition}\lbl{defn-D-equiv}
 Consider a morphism $f\:P\to Q$ in $\POSet$, and note that
 $c_Qf=c_P\:P\to e$.  For any derivator $\CC$ and any $X,Y\in\CC(Q)$
 we therefore get a map
 \[ f^* \: \CC(Q)(c_Q^*X,c_Q^*Y) \to \CC(P)(c_P^*X,c_P^*Y). \]
 We say that $f$ is a \emph{$\CD$-equivalence} if this map is
 bijective for all $\CC$, $X$ and $Y$.  We also say that $P$ is
 \emph{$\CD$-contractible} if the map $c_P\:P\to e$ is a
 $\CD$-equivalence, or equivalently the functor
 \[ c_P^* \: \CC(e) \to \CC(P) \]
 is full and faithful.
\end{definition}

\begin{remark}
 Groth uses the term \emph{homotopy contractible} rather than
 \emph{$\CD$-contractible}. 
\end{remark}

\begin{proposition}\lbl{prop-D-equiv-inv}
 If $[f]=[g]$ in $\pi_0(\POSet(P,Q))$ then
 \[ f^* = g^* \: \CC(Q)(c_Q^*X,c_Q^*Y) \to \CC(P)(c_P^*X,c_P^*Y). \]
 Thus, $f$ is a $\CD$-equivalence iff $g$ is a $\CD$-equivalence.
\end{proposition}
\begin{proof}
 We can reduce easily to the case where $f\leq g$.
 As $\CC\:\POSet^{\text{op}}\to\CAT$ is a strict $2$-functor, the
 following diagram of categories and functors must commute on the
 nose:
 \begin{center}
  \begin{tikzcd}[column sep=large]
   \POSet(P,Q)\tm\POSet(Q,e)
     \arrow[rr,"\text{compose}"] \arrow[d] &&
   \POSet(P,e) \arrow[d] \\
   {[\CC(Q),\CC(P)]\tm[\CC(e),\CC(Q)]}
     \arrow[rr,"\text{compose}"'] &&
   {[\CC(e),\CC(P)].} 
  \end{tikzcd}  
 \end{center}
 The inequality $f\leq g$ gives a morphism $(f,c_Q)\to(g,c_Q)$ in the
 category $\POSet(P,Q)\tm\POSet(Q,e)$, and this becomes the identity
 morphism of $c_P$ in $\POSet(P,e)$.  The claim now follows by chasing
 the diagram.
\end{proof}

\begin{corollary}\lbl{cor-D-equiv-inv}
 If $f\:P\to Q$ is a strong homotopy equivalence, then it is a
 $\CD$-equivalence.  In particular:
 \begin{itemize}
  \item[(a)] If $f$ has a left or right adjoint, then it is a
   $\CD$-equivalence.  
  \item[(b)] If $P$ is strongly contractible, then it is
   $\CD$-contractible.  \qed 
 \end{itemize}
\end{corollary}

The following definitions are taken
from~\cite{groth:mayervietoris}*{Section 3}:
\begin{definition}\lbl{defn-cofinal}\leavevmode
 \begin{itemize}
  \item[(a)] We say that a map $f\:Q\to P$ is \emph{homotopy final}
   if the natural map
   \[ \hocolim_Qf^*(X) = (c_Q)_!f^*(X) = (c_P)_!f_!f^*(X) \to 
        (c_P)_!(X) = \hocolim_P(X) 
   \]
   is an isomorphism for all derivators $\CC$ and all objects
   $X\in\CC(P)$.      
  \item[(b)] Dually, we say that a map $f\:Q\to P$ is \emph{homotopy
   cofinal} if the natural map
   \[ \holim_P(X) = (c_P)_*(X) \to (c_P)_*f_*f^*(X) =
       (c_Q)_*f^*(X) = \holim_Qf^*(X) 
   \]
   is an isomorphism for all derivators $\CC$ and all objects
   $X\in\CC(P)$.      
 \end{itemize}
\end{definition}

We do not really need the following result, but it helps to clarify
the relationship between our definitions.
\begin{proposition}\lbl{prop-final-equiv}
 If $f\:Q\to P$ is homotopy final or homotopy cofinal then it is a
 $\CD$-equivalence. 
\end{proposition}
\begin{proof}
 We will just treat the final case, as the other case is dual.  Taking
 $X=c_P^*(U)$ in the definition, we see that the natural map
 $(c_Q)_!c_Q^*(U)\to(c_P)_!c_P^*(U)$ is an isomorphism.  This gives an
 isomorphism
 \[ \CC(e)((c_P)_!c_P^*(U),V) \simeq \CC(e)((c_Q)_!c_Q^*(U),V) \]
 for any $V$.  By adjunction, we get an isomorphism
 \[ \CC(P)(c_P^*(U),c_P^*(V)) \simeq \CC(Q)(c_Q^*(U),c_Q^*(V)). \]
 We leave it to the reader to check that this is just $f^*$, as
 required. 
\end{proof}

\begin{proposition}\lbl{prop-cofinal}
 The map $f$ is homotopy final iff $p/f$ is strongly contractible for
 all $p$, and this holds if $f$ has a left adjoint.  Dually, $f$ is
 homotopy cofinal iff $f/p$ is strongly contractible for all $p$, and
 this holds if $f$ has a right adjoint.
\end{proposition}
\begin{proof}
 See~\cite{groth:mayervietoris}*{Corollary 3.13} and surrounding
 discussion. 
\end{proof}

\begin{proposition}\lbl{prop-cofinal-square}
 Consider a commutative square
 \begin{center}
  \begin{tikzcd}
   P \arrow[r,"t"] \arrow[d,"u"'] & Q \arrow[d,"v"] \\
   R \arrow[Rightarrow, ru,"\alpha" ] \arrow[r,"w"'] & S,
  \end{tikzcd}
 \end{center}
 and the resulting Beck-Chevalley transform $\al_*\:w^*v_*\to u_*t^*$
 between morphisms $\CC^Q\to\CC^R$.  For any $r\in R$ we have comma
 posets $r/u$ and $w(r)/v$, and using $t$ and $\alpha$ we can produce a 
 morphism $t_r\:r/u\to w(r)/v$.  If this is homotopy cofinal for all
 $r$, then $\al$ is an isomorphism.
\end{proposition}
\begin{proof}
 We must show that for any $T$ and any $X\in\CC^Q(T)=\CC(Q\tm T)$, the
 map $(\al_*)_X\:w^*v_*X\to u_*t^*X$ is an isomorphism in $\CC(R\tm T)$.
 After replacing $\CC$ by $\CC^T$ we can assume that $T=e$, so
 $X\in\CC(Q)$ and $(\al_*)_X$ is a morphism in $\CC(R)$.  It will suffice
 to show that $i_r^*(\al_*)_X$ is an isomorphism for all $r\in R$.  The
 Kan formula expresses the source and target of $i_r^*(\al_*)_X$ as
 homotopy limits over comma posets.  In more detail, there is a
 projection $\pi_r\:w(r)/v\to Q$, and the source of $i_r^*(\al_*)_X$ is
 the homotopy limit of $\pi_r^*X$, whereas the target is the homotopy
 limit of $t_r^*\pi_r^*X$.  The homotopy cofinality condition says
 that the natural map between these is an isomorphism.
\end{proof}

\begin{definition}\lbl{defn-subdiv}
 Recall that a subset $\sg\sse P$ is a \emph{chain} if the induced
 order on $\sg$ is total.  If $\sg$ is a chain, we write
 $\dim(\sg)=|\sg|-1$.  We put
 \begin{align*}
  s(P) &= \{ \text{ nonempty chains } \sg \sse P\} \\
  s_d(P) &= \{\sg\in s(P)\st\dim(\sg)=d\} \\
  s_{\leq d}(P) &= \{\sg\in s(P)\st\dim(\sg)\leq d\}.
 \end{align*}
 Note that every nonempty chain $\sg$ has a largest element, which we
 denote by $\max(\sg)$.  We order $s(P)$ by inclusion, which ensures
 that $\max\:s(P)\to P$ is a morphism of posets.
\end{definition}

This construction gives a functor $s\:\POSet\to\POSet$, and
$\max\:s(P)\to P$ is natural.  However, if $f\leq g$ then it is
typically not the case that $s(f)\leq s(g)$.  This makes the following
proof more complex than one might expect.

\begin{proposition}\lbl{prop-subdiv-functor}
 If $[f]=[g]$ in $\pi_0(\POSet(P,Q))$, then $[s(f)]=[s(g)]$.  Thus,
 $s$ induces an endofunctor of the strong homotopy category.
\end{proposition}
\begin{proof}
 We can easily reduce to the case where $f\leq g$.  We then choose a
 minimal element $p_1$ in $P$, then a minimal element $p_2$ in
 $P\sm\{p_1\}$ and so on, giving an enumeration
 $P=\{p_1,\dotsc,p_{m-1}\}$ say.  We define
 $\phi\:P\to[m]=\{0,\dotsc,m\}$ by $\phi(p_i)=i$ (so $\phi$ is
 injective and monotone, and $0$ and $m$ are not in the image).  Then
 for $0\leq k\leq m$ we define $u_k,v_k\:s(P)\to s(Q)$ by
 \begin{align*}
  u_k(\sg) &= \{f(p)\st p\in\sg,\; \phi(p)<k\} \cup
              \{g(p)\st p\in\sg,\; \phi(p)\geq k\} \\
  v_k(\sg) &= \{f(p)\st p\in\sg,\; \phi(p)\leq k\} \cup
              \{g(p)\st p\in\sg,\; \phi(p)\geq k\}.
 \end{align*}
 We find that $u_k(\sg)$ and $v_k(\sg)$ are nonempty chains in $Q$, so
 we really do have maps $u_k,v_k\:s(P)\to s(Q)$ as advertised.  It is
 also clear that when $\sg\sse\tau$ we have $u_k(\sg)\sse u_k(\tau)$
 and $v_k(\sg)\sse v_k(\tau)$, so $u_k$ and $v_k$ are maps of posets.
 Next, we find that $u_k,u_{k+1}\leq v_k$, which implies that all the
 maps $u_k$ and $v_k$ lie in the same path component.  Finally, we
 have $u_0=s(g)$ and $u_m=s(f)$, so $[s(f)]=[s(g)]$ as claimed.
\end{proof}

\begin{lemma}\lbl{lem-subdiv-cofinal}
 The map $\max\:s(P)\to P$ is homotopy cofinal (and so is a
 $\CD$-equivalence).
\end{lemma}
\begin{proof}
 Fix $p\in P$; it will suffice to show that the poset
 \[ U = \max/p=\{\sg\in sP\st \max(\sg)\leq p\} \]
 is strongly contractible.  Put 
 \[ V = \{\sg\in sP\st \max(\sg)= p\} =
        \{\sg\in U\st p\in\sg\}\sse U.
 \]
 As $\{p\}$ is smallest in $V$, we see that $V$ is strongly
 contractible.  We can define a poset map $t\:U\to V$ by
 $t(\tau)=\tau\cup\{p\}$, and we find that this is left adjoint to the
 inclusion $V\to U$, so the inclusion is a strong homotopy equivalence
 by Remark~\ref{rem-adjoint-strong}.  It follows that $U$ is also
 strongly contractible.
\end{proof}

\begin{lemma}\lbl{lem-max-equiv}
 The map $|\max|\:|s(P)|\to|P|$ is a homotopy equivalence.
\end{lemma}
\begin{proof}
 We can define a barycentric subdivision map
 $\bt\:|s(P)|\to\Map(P,[0,1])$ by
 \[ \bt(w)(p) = \sum_{p\in\sg} |\sg|^{-1} w(\sg). \]
 It is well-known that this gives a homeomorphism $|s(P)|\to|P|$.  A
 typical simplex $\tau$ of $s(P)$ is a chain $\{\sg_0,\dotsc,\sg_d\}$
 where each $\sg_i$ is a nonempty chain in $P$ and
 $\sg_0\subset\dotsb\subset\sg_d$.  It is easy to see that both $\bt$
 and $|\max|$ send $|\tau|$ into $|\sg_d|$, so the straight-line
 homotopy between $\bt$ and $|\max|$ stays within $|P|$.  This means
 that $|\max|$ is homotopic to a homeomorphism, and so is a homotopy
 equivalence. 
\end{proof}

\begin{definition}\lbl{defn-weak-homotopy}
 The \emph{weak homotopy category} of finite posets is obtained from
 the strong homotopy category by inverting the maps
 $\max\:s(P)\to P$.  Thus every morphism $P\to Q$ in the weak homotopy
 category can be represented as $f\circ\max^{-r}$ for some
 $f\:s^r(P)\to Q$, with $f\circ\max^{-r}=g\circ\max^{-s}$ iff
 $[f\circ\max^{s+t}]=[g\circ\max^{r+t}]$ in
 $\pi_0(\POSet(s^{r+s+t}(P),Q))$ for sufficiently large $t$.
\end{definition}

\begin{remark}
 Using Lemma~\ref{lem-max-equiv} we see that geometric realisation
 gives a functor from the weak homotopy category to the homotopy
 category of finite simplicial complexes.  A slight variant of the
 standard simplicial approximation theorem shows that this functor is
 an equivalence.  However, we do not need this, so we will not spell
 out the details.
\end{remark}

We will also use some theory of total fibres.  We recall some basic
ideas.
\begin{definition}\lbl{defn-tfib}
 Let $R$ be a finite set, and let $PR$ be the poset of subsets of $R$.
 Let $j\:e\to PR$ correspond to $\emptyset\in PR$, and put
 $P'R=PR\sm\{\emptyset\}$, and let $i\:P'R\to PR$ be the inclusion.
 As $j$ is a sieve, Proposition~\ref{prop-sieve} gives us a functor
 $j^!\:\CC(PR)\to\CC(e)$ that is right adjoint to $j_*$.  We also
 write $\tfib(X)$ or $\tfib_R(X)$ for $j^!(X)$, and call this the
 \emph{total fibre} of an object $X\in\CC(PR)$.  We also say that $X$
 is \emph{cartesian} if it is in the essential image of
 $i_*\:\CC(P'R)\to\CC(R)$. 
\end{definition}

\begin{lemma}\lbl{lem-cartesian-tfib}
 $X$ is cartesian iff $\tfib(X)=0$.
\end{lemma}
\begin{proof}
 We use Proposition~\ref{prop-recollement}.  Part~(a) of that result
 includes the relation $j^!i_*=0$, which means that $\tfib$ is
 trivial on cartesian objects.  For the converse, part~(b) tells us
 that the fibre of the unit map $X\to i_*i^*(X)$ is $j_*j^!(X)$, so if
 $j^!(X)=0$ we see that $X\simeq i_*i^*(X)$ and so $X$ is cartesian.
\end{proof}

\begin{lemma}\lbl{lem-tfib-fibre}
 $\tfib(X)$ is the fibre of the natural map
 $X_\emptyset\to\holim_{P'R}j^*X$. 
\end{lemma}
\begin{proof}
 Let $c$ and $c'$ be the maps $PR\to e$ and $P'R\to e$, so $cj=1$ and
 $ci=c'$.  Proposition~\ref{prop-recollement}(b) gives us a
 distinguished triangle $j_*j^!X\to X\to i_*i^*X$.  
 The functor $c_*\:\CC(PR)\to\CC(e)$ is exact
 by~\cite{groth:derpointstable}*{Corollary 4.19}.  We can therefore
 apply it to get a distinguished triangle
 $c_*j_*j^!X\to c_*X\to c_*i_*i^*X$.  As $cj=1$, the first term is
 just $j^!X=\tfib(X)$.  As $j$ is left adjoint to $c$, we can identify
 $c_*$ with $j^*$ (as in~\cite{groth:derpointstable}*{Lemma 1.23}), so
 the middle term is $j^*X=X_\emptyset$.  As $ci=c'$, the last term is
 $c'_*i^*X=\holim_{P'R}i^*X$.  
\end{proof}

We now want to discuss another description of $\tfib(X)$.  
Suppose that $r\in R$, and put $R_0=R\sm\{r\}$.  Define
$k_-,k_+\:PR_0\to PR$ by $k_-(T)=T$ and $k_+(T)=T\cup\{r\}$.  As
$k_-\leq k_+$ we have a natural map $k_-^*X\to k_+^*X$ in $\CC(R_0)$,
and thus a map $\tfib(k_-^*X)\to\tfib(k_+^*X)$.

\begin{proposition}\lbl{prop-tfib-split}
 There is a natural distinguished triangle
 \[ \tfib(X) \to \tfib(k_-^*X) \to \tfib(k_+^*X) \to \Sg\tfib(X). \]
\end{proposition}
\begin{proof}
 We can regard $X$ as an object of $\CC^{PR_0}(P\{r\})$.  From this
 point of view $k_-^*X$ is just $X_{\emptyset}$, and $P'\{r\}\simeq e$
 so $k_+^*X$ is just $\holim_{P'\{r\}}j^*X$.  We also see that
 $k_-^!X=\tfib_{\{r\}}(X)$.  Lemma~\ref{lem-tfib-fibre} therefore
 gives us a distinguished triangle $k_-^!X\to k_-^*X\to k_+^*X$.  Now
 let $i_0\:e\to PR_0$ correspond to $\emptyset$, and apply the exact
 functor $i_0^!$ to the above triangle.  As $k_-i_0=i\:e\to PR$, the
 first term is $i^!X=\tfib(X)$.  The other two terms are
 $\tfib(k_-^*X)$ and $\tfib(k_+^*X)$, as required. 
\end{proof}

\begin{proposition}\lbl{prop-loop}
 Let $t\:e\to PR$ correspond to the element $R\in PR$.  Then there is
 a natural isomorphism $\tfib(t_!X)\simeq\Om^{|R|}(X)$.
\end{proposition}
\begin{proof}
 We can split $R$ as $R_0\amalg\{r\}$ as before, and let
 $t_0\:e\to PR_0$ correspond to $R_0\in PR_0$, and put $n_0=|R_0|$.
 It is then easy to identify $k_+^*t_!(X)$ with $(t_0)_!(X)$, so by
 induction we have $\tfib(k_+^*(t_!(X)))=\Om^{n_0}(X)$.  On the other
 hand, we can check that $k_-^*(t_!(X))=0$, so
 $\tfib(k_-^*(t_!(X)))=0$.  Now Proposition~\ref{prop-tfib-split}
 tells us that $\tfib(t_!(X))$ is the fibre of $0\to\Om^{n_0}(X)$,
 which is $\Om^{n_0+1}(X)$ as required.      
\end{proof}

The following proposition is a derivator version of a standard result
about homotopy limits in simplicial or topological categories.
\begin{proposition}\lbl{prop-cosimp-replacement}
 Let $n$ be the maximum length of any chain in $P$.  Then for all
 stable derivators $\CC$ and objects $X\in\CC(P)$ there is a natural
 tower
 \[ \holim_P(X) =
     T^n(X) \to T^{n-1}(X) \to \dotsb \to T^0(X) \to T^{-1}(X)
     = 0
 \]
 and natural distinguished triangles
 \[ \bigoplus_{\sg\in s_d(P)} \Om^dX_{\max(\sg)} \to
     T^d(X) \to T^{d-1}(X).
 \]
\end{proposition}
\begin{proof}
 Put $Y=\max^*(X)\in\CC(s(P))$.  Lemma~\ref{lem-subdiv-cofinal}
 identifies $\holim_P(X)$ with $\holim_{s(P)}(Y)$, so we will work
 with $Y$ from now on.
 
 Let $j_d\:s_{\leq d}(P)\to s(P)$ be the inclusion, and put
 $T^d(X)=\holim j_d^*(Y)$.  Note that
 $T^n(Y)=\holim_{s(P)}Y=\holim_PX$.

 Now fix $d$ and consider the object $Z=j_{\leq d}^*(Y)$ and the
 inclusions $j\:s_{\leq(d-1)}(P)\to s_{\leq d}(P)$ and
 $i\:s_d(P)\to s_{\leq d}(P)$.  Proposition~\ref{prop-recollement}
 gives a distinguished triangle $i_!i^*(Z)\to Z\to j_*j^*(Z)$.  If we
 let $c$ be the map $s_{\leq d}(P)\to e$ and apply $c_*$, we get a
 distinguished triangle $c_*i_!i^*(Z)\to T^d(X)\to T^{d-1}(X)$.  Now
 note that the order on $s_d(P)$ is discrete: for $\sg,\tau\in s_d(P)$
 we can only have $\sg\leq\tau$ if $\sg=\tau$.  Because of this, we
 see that $\CC(s_d(P))\simeq\prod_{\sg\in s_d(P)}\CC(e)$.  Because of
 this, we can write $i^*(Z)$ as a coproduct of objects $W(\sg)$, where
 $W(\sg)_\tau=0$ for $\tau\neq\sg$, and
 $W(\sg)_\sg=Z_\sg=Y_\sg=X_{\max(\sg)}$.  It will now suffice to
 identify $c_*i_!W(\sg)$ with $\Om^dY_\sg$.  Here $i$ is a cosieve,
 and it follows that for $\tau\in s_{\leq d}(P)$ we have
 \[ (i_!W(\sg))_\tau = \begin{cases}
                        Y_\sg & \text{ if } \tau = \sg \\
                        0 & \text{ otherwise }.
                       \end{cases}
 \]
 The poset $\{\tau\in s(P)\st\tau\sse\sg\}$ is naturally identified
 with $P'(\sg)$.  Let $k\:P'(\sg)\to s_{\leq d}(P)$ be the inclusion,
 which is a sieve.  As the support of $i_!W(\sg)$ is contained in the
 image of $k$, we see from Proposition~\ref{prop-sieve} that the unit
 map $i_!W(\sg)\to k_*k^*i_!W(\sg)$ is an isomorphism.  It follows
 that $c_*i_!W(\sg)$ is the homotopy limit of $k^*i_!W(\sg)$.  Here it
 is easy to see that the object $k^*i_!W(\sg)\in\CC(P'(\sg))$ is the
 restriction to $P'(\sg)$ of the object $t_!W(\sg)$ appearing in
 Proposition~\ref{prop-loop}.  That proposition tells us that the
 total fibre of $t_!W(\sg)$ is $\Om^{d+1}Y_\sg$.  On the other hand,
 Lemma~\ref{lem-tfib-fibre} tells us that the total fibre is the same
 as the fibre of the natural map from $(t_!W(\sg))_\emptyset=0$ to
 $\holim k^*i_!W(\sg)$, which is $\Om\holim k^*i_!W(\sg)$.  As we are
 working with stable derivators we know that $\Om$ is an equivalence
 of categories, so $\holim k^*i_!W(\sg)=\Om^dY_\sg$ as required.
\end{proof}

\begin{corollary}\lbl{cor-simp-replacement}
 Let $n$ be the maximum length of any chain in $P$.  Then for all
 stable derivators $\CC$ and objects $X\in\CC(P)$ there is a natural
 diagram
 \[ 0 = T_{-1}(X) \to T_0(X) \to \dotsb \to T_n(X) = 
      \hocolim_P(X) 
 \]
 and natural distinguished triangles
 \[ T_{d-1}(X) \to T_d(X) \to
     \bigoplus_{\sg\in s_d(P)} \Sg^dX_{\max(\sg)}.
 \]
\end{corollary}
\begin{proof}
 Apply the proposition to the dual derivator.
\end{proof}

\begin{proposition}\lbl{prop-preserve-products}
 For any map $u\:Q\to P$ of finite posets, the functors $u^*$, $u_!$
 and $u_*$ all preserve arbitrary products and coproducts.  
\end{proposition}
\begin{proof}
 The functors $u^*$ and $u_!$ have right adjoints, so they preserve
 coproducts.  The functors $u^*$ and $u_*$ have left adjoints, so they
 preserve products.  The key point is to prove that $u_*$ preserves
 coproducts.  Consider a family of objects $X_\al\in\CC(Q)$, and the
 resulting map
 $f\:\bigoplus_\al u_*(X_\al)\to u_*\left(\bigoplus_\al X_\al\right)$.
 We want to prove that $f$ is an isomorphism, and it will suffice to
 show that $i_p^*(f)$ is an isomorphism for all $p\in P$.  We have
 already observed that $i_p^*$ preserves coproducts, and we have a Kan
 formula expressing $i_p^*u_*$ as a homotopy limit over $p/u$.  It
 will therefore suffice to show that all homotopy limit functors
 preserve coproducts.  Note that the functor $\Om\:\CC(e)\to\CC(e)$ is
 an equivalence of categories, so it certainly preserves coproducts.
 It follows that all functors of the form $X\mapsto\Om^dX_q$ also
 preserve coproducts.  It then follows by induction that the functors
 $T^d$ in Proposition~\ref{prop-cosimp-replacement} all preserve
 coproducts.  By taking $d$ sufficiently large, we see that homotopy
 limits preserve coproducts as required.  This completes the proof
 that $u_*$ preserves coproducts, and we can apply that to the dual
 derivator to see that $u_!$ preserves products.
\end{proof}

\section{Thick subderivators}
\label{sec-thick}

\begin{definition}\lbl{defn-subderivator}
 Let $\CC$ be a stable derivator.  By an \emph{thick subderivator}
 $\CE\sse\CC$ we mean a system of full subcategories
 $\CE(P)\sse\CC(P)$ for all $P$, such that:
 \begin{itemize}
  \item[(a)] Each category $\CE(P)$ is closed under finite coproducts
   (including the empty coproduct, so $0\in\CE(P)$).
  \item[(b)] Whenever $X\xra{j}Y\xra{q}X$ in $\CC(P)$ with $qj=1$, if
   $Y\in\CE(P)$ then $X\in\CE(P)$.  In other words, $\CE$ is closed
   under retracts.  In particular, if $X\simeq Y$ and $Y\in\CE(P)$
   then $X\in\CE(P)$.
  \item[(c)] For any morphism $u\:P\to Q$ of finite posets, we have
   $u^*\CE(Q)\sse\CE(P)$ and $u_!\CE(P)\sse\CE(Q)$ and
   $u_*\CE(P)\sse\CE(Q)$.  More briefly, we say that the functors
   $u^*$, $u_*$ and $u_!$ preserve $\CE$.
 \end{itemize}
\end{definition}

\begin{definition}\lbl{defn-thick-gamma}
 Let $\CC$ be a stable derivator, and let $\CE_1$ be a thick
 subcategory of $\CC(e)$.  For any finite poset $P$ and any $p\in P$
 we have a corresponding map $i_p\:e\to P$.
 \begin{itemize}
  \item[(a)] We put 
   \[ (\gm_0\CE_1)(P) =
      \{X\in \CC(P) \st i_p^*(X) \in \CE_1 \text{ for all } p \in P\} \]
  \item[(b)] We let $(\gm_1\CE_1)(P)$ denote the smallest thick
   subcategory of $\CC(P)$ containing $(i_p)_!(\CE_1)$ for all $p$.
  \item[(c)] We let $(\gm_2\CE_1)(P)$ denote the smallest thick
   subcategory of $\CC(P)$ containing $(i_p)_*(\CE_1)$ for all $p$.
 \end{itemize}
\end{definition}

\begin{theorem}\lbl{thm-thick-subcats}
 The subcategories $(\gm_i\CE_1)(P)$ are the same for $i=0,1,2$ (so we
 will just write $(\gm\CE_1)(P)$ in future).  The map $\gm$ gives a
 bijection from thick subcategories of $\CC(e)$ to thick subderivators
 of $\CC$.  Moreover, if $\CE$ is a thick subderivator of $\CC$, then
 $\CE(P)$ is a thick subcategory of $\CC(P)$ for all $P$.
\end{theorem}

The proof will be given after some lemmas.

\begin{remark}\lbl{rem-finite-needed}
 It is important here that our derivators are indexed on finite posets
 rather than more general categories; the theorem would not be true
 without some restriction of this kind.  In particular it would fail
 for derivators indexed by finite groups; this is related to the fact
 that classifying spaces of finite groups are infinite complexes.
\end{remark}

\begin{lemma}\lbl{lem-thick-thick}
 If $\CE\sse\CC$ is a thick subderivator, then $\CE(P)$ is a thick
 subcategory of $\CC(P)$ for all $P$, and
 \[ (\gm_1\CE(e))(P) \cup (\gm_2\CE(e))(P) \sse \CE(P) \sse
     (\gm_0\CE(e))(P).
 \] 
\end{lemma}
\begin{proof}
 The triangulation of $\CC(P)$ is defined in terms of operations of the
 form $u^*$, $u_!$, $u_*$, $u^!$ and $u^?$.  However, the operations
 $u^!$ and $u^?$ are themselves defined in terms of $v^*$, $v_*$ and
 $v_!$ for various auxiliary morphisms $v$, as discussed
 in~\cite{groth:derpointstable}*{Section 2}.  From this it follows
 that $\CE(P)$ is closed under the suspension functor and its inverse,
 and under cofibrations, so it is a thick subcategory.  From the
 definitions we also know that the functors $(i_p)_!$ preserve $\CE$,
 so $\CE(P)$ contains the generators of $(\gm_1\CE(e))(P)$, so it
 contains the whole of $(\gm_1\CE(e))(P)$.  It also contains
 $(\gm_2\CE(e))(P)$ by essentially the same argument.  On the other
 hand, the functors $i_p^*$ also preserve $\CE$, and this means that
 $\CE(P)\sse(\gm_0\CE(e))(P)$. 
\end{proof} 

\begin{lemma}\lbl{lem-gmo-subderivator}
 If $\CE_1$ is a thick subcategory of $\CC(e)$, then $\gm_0\CE_1$ is a
 thick subderivator of $\CC$, with $(\gm_0\CE_1)(e)=\CE_1$.
\end{lemma}
\begin{proof}
 Suppose we have a morphism $u\:Q\to P$ of finite posets.  It is clear
 from the definitions that $u^*((\gm_0\CE_1)(P))\sse (\gm_0\CE_1)(Q)$,
 or more briefly that $u^*$ preserves $\gm_0\CE_1$.  We claim that the
 functors $u_*$ and $u_!$ also preserve $\gm_0\CE_1$.  In the case
 $Q=e$, this follows easily from
 Proposition~\ref{prop-cosimp-replacement} and
 Corollary~\ref{cor-simp-replacement}.  We can use the Kan formulae to
 deduce the general case from the case $Q=e$.  It is also easy to
 check that $\gm_0\CE_1$ is closed under retracts, so $\gm_0\CE_1$ is
 a thick subderivator of $\CC$.  The relation
 $(\gm_0\CE_1)(e)=\CE_1$ is clear.
\end{proof}

\begin{corollary}\lbl{cor-gm-inc}
 If $\CE_1$ is a thick subcategory of $\CC(e)$, then
 $(\gm_1\CE_1)(P)\cup(\gm_2\CE_1)(P)\sse(\gm_0\CE_1)(P)$ for all $P$. 
\end{corollary}
\begin{proof}
 Lemma~\ref{lem-gmo-subderivator} allows us to apply
 Lemma~\ref{lem-thick-thick} to the case $\CE=\gm_0\CE_1$.
\end{proof}

\begin{lemma}\lbl{lem-gm-preserved}
 If $\CE_1$ is a thick subcategory of $\CC(e)$, then all functors
 $u_!$ preserve $\gm_1\CE_1$, and all functors $u_*$ preserve
 $\gm_2\CE_1$. 
\end{lemma}
\begin{proof}
 Fix a map $u\:P\to Q$, and put
 $\CU=\{X\in\CC(P)\st u_*(X)\in (\gm_1\CE_1)(Q)\}$.  It is easy to see
 that $\CU$ is thick.  As $ui_p=i_{u(p)}\:e\to P$ we see that
 $u_!(i_p)_!=(i_{u(p)})_!$, and it follows that
 $(i_p)_!(\CE_1)\sse\CU$ for all $p\in P$.  It follows that
 $(\gm_1\CE_1)(P)\sse\CU$, which means that $u_*$ preserves
 $\gm_1\CE_1$ as required.  Dually, we see that $u_*$ preserves
 $\gm_2\CE_1$.  
\end{proof}

\begin{lemma}\lbl{lem-gm-eq}
 If $\CE_1$ is a thick subcategory of $\CC(e)$, then
 $\gm_0\CE_1=\gm_1\CE_1=\gm_2\CE_1$.
\end{lemma}
\begin{proof}
 We will write $\Gm_i=\gm_i\CE_1$ for brevity.  It will be enough to
 prove that $\Gm_0=\Gm_1$, as duality then gives $\Gm_0=\Gm_2$.  We
 already know from Corollary~\ref{cor-gm-inc} that 
 $\Gm_1(P)\sse\Gm_0(P)$ for all $P$, so it will suffice
 to prove that $\Gm_0(P)\sse\Gm_1(P)$.  This is clear if
 $P$ is empty.  If $P$ is nonempty, we can choose a minimal element
 $a\in P$ and put $Q=P\sm\{a\}$.  Let $j\:e\to P$ correspond to $a$,
 and let $i\:Q\to P$ be the inclusion.  Consider an object
 $Y\in\Gm_0(P)$; we must show that $Y\in\Gm_1(P)$.  Put $X=j_!j^*(Y)$,
 and let $Z$ be the cofibre of the counit map $X\to Y$.  From
 Lemma~\ref{lem-embedding} we have $j^*j_!=1$ and so $j^*Z=0$, so the
 support of $Z$ is contained in $i(Q)$.  As $i$ is a cosieve, we see
 that $Z\simeq i_!i^*(Z)$.  Now $i^*(Z)\in\Gm_0(Q)$, and we can assume
 by induction that $\Gm_0(Q)\sse\Gm_1(Q)$, so
 $Z\in i_!\Gm_1(Q)\sse\Gm_1(P)$.  From the definitions we also have
 $j^*(X)\in\CE_1$ and $X\in\Gm_1(P)$.  As $\Gm_1(P)$ is thick and
 contains $X$ and $Z$, it also contains $Y$ as required.  
\end{proof}

\begin{proof}[Proof of Theorem~\ref{thm-thick-subcats}]
 First suppose that we start with a thick subcategory
 $\CE_1\sse\CC(e)$.  Lemma~\ref{lem-gm-eq} tells us that the
 $\gm_i\CE_1$ are all the same, so we can just write $\gm\CE_1$.
 Lemma~\ref{lem-gmo-subderivator} tells us that this is a thick
 subderivator, with $(\gm\CE_1)(e)=\CE_1$.

 Suppose instead that we start with a thick subderivator
 $\CE\sse\CC$, and we put $\CE_1=\CE(e)$.  Lemma~\ref{lem-thick-thick}
 tells us that $\CE(P)$ is a thick subcategory of $\CC(P)$ for all
 $P$, and in particular that $\CE_1$ is a thick subcategory of
 $\CC(e)$.  We can therefore apply Lemma~\ref{lem-gm-eq} to $\CE_1$
 and combine this with Lemma~\ref{lem-thick-thick} to see that
 $\CE=\gm\CE_1$. 
\end{proof}

\begin{corollary}\lbl{cor-thick-union}
 Let $\CE$ be a thick subderivator of $\CC$, and let $X$ be an object
 of $\CC(P)$.  Suppose that $P=\bigcup_iP_i$ for some family of
 subposets $P_i$.  Then $X$ lies in $\CE(P)$ iff $X|_{P_i}\in\CE(P_i)$
 for all $i$.
\end{corollary}
\begin{proof}
 The identity $\CE=\gm\CE(e)$ means that $X\in\CE(P)$ iff
 $i_p^*(X)\in\CE(e)$ for all $p\in P$.  Similarly,
 $X|_{P_i}\in\CE(P_i)$ iff $i_p^*(X)\in\CE(e)$ for all
 $p\in P_i$.  The claim is clear from this.
\end{proof}

\begin{definition}\lbl{defn-localising}
 Let $\CT$ be a triangulated category with arbitrary coproducts.
 Recall that a \emph{localising subcategory} of $\CT$ is a thick
 subcategory that is closed under arbitrary coproducts.  Similarly, if
 $\CC$ is a stable derivator, a \emph{localising subderivator} is a
 thick subderivator $\CE\sse\CC$ such that the subcategory
 $\CE(P)\sse\CC(P)$ is closed under arbitrary coproducts for all $P$.
\end{definition}

\begin{proposition}\lbl{prop-loc-loc}
 The map $\gm$ gives a bijection between localising subcategories of
 $\CC(e)$ and localising subderivators of $\CC$, with inverse
 $\CE\mapsto\CE(e)$. 
\end{proposition}
\begin{proof}
 Firstly, if $\CE$ is a localising subderivator of $\CC$, then it is
 immediate from the definitions that $\CE(e)$ is a localising
 subcategory of $\CC(e)$.

 In the opposite direction, suppose that $\CE_1$ is a localising
 subcategory of $\CC(e)$.  Let $(X_\al)$ be a family of objects of
 $(\gm_0\CE_1)(P)$, so $i_p^*(X_\al)\in\CE_1$ for all $\al$ and all
 $p\in P$.  As $i_p^*$ has a right adjoint we see that it preserves
 coproducts, so
 $i_p^*(\bigoplus_\al X_\al)=\bigoplus_\al i_p^*(X_\al)\in\CE_1$.  As
 this holds for all $p$ we see that
 $\bigoplus_\al X_\al\in(\gm_0\CE_1)(P)$.  This shows that $\gm\CE_1$
 is a localising subderivator, as required.
\end{proof}

\begin{lemma}\lbl{lem-gm-left}
 Let $\CC$ be a stable derivator, and let $\CE(P)$ be a thick
 subcategory of $\CC(P)$ for all $P$.  Suppose that for every
 $u\:Q\to P$, the functors $u_!$ and $u^*$ preserve $\CE$.  Then
 $\CE=\gm\CE(e)$, so in particular $\CE$ is a thick subderivator.
\end{lemma}
\begin{proof}
 Put $\CE'=\gm\CE(e)=\gm_0\CE(e)=\gm_1\CE(e)$.  We now claim that
 $\CE(P)\sse\CE'(P)$ for all $P$.  Using the description
 $\CE'=\gm_0\CE(e)$, this reduces to the claim that
 $i_p^*\CE(P)\sse\CE(e)$ for all $p$, which is true because $i_p^*$
 preserves $\CE$ by assumption.  In the opposite direction, we know
 that the functors $(i_p)_!$ preserve $\CE$, which means that $\CE(P)$
 contains all the generators of $\CE'(P)=(\gm_1\CE(e))(P)$.  As
 $\CE(P)$ is assumed to be thick, it follows that
 $\CE'(P)\sse\CE(P)$. 
\end{proof}

\begin{definition}\lbl{defn-compact}
 Let $\CT$ be a triangulated category with coproducts.  We say that an
 object $X\in\CT$ is \emph{compact} if the natural map
 $\bigoplus_\al[X,Y_\al]\to[X,\bigoplus_\al Y_\al]$ is an isomorphism
 for every family of objects $Y_\al$.  We write $\CT_c$ for the full
 subcategory of compact objects (which is easily seen to be thick).
 If $\CT=\CC(P)$ for some derivator $\CC$, then we will write
 $\CC_c(P)$ rather than $\CC(P)_c$.  We say that $\CT$ is
 \emph{compactly generated} if
 \begin{itemize}
  \item[(a)] The category $\CT_c$ is essentially small (so there is a
   skeleton that has a set of objects, rather than a proper class);
   and
  \item[(b)] $\CT$ is the only thick subcategory of $\CT$ that is
   closed under arbitrary coproducts and contains $\CT_c$. 
 \end{itemize}
\end{definition}

\begin{lemma}\lbl{lem-ess-small}
 Let $\CT$ be a triangulated category, let $\CG$ be a set of objects
 of $\CT$, and let $\CU$ be the smallest thick subcategory containing
 $\CG$.  Then $\CU$ is essentially small.
\end{lemma}
\begin{proof}
 Define full subcategories $\CU_n$ as follows.  Start with
 $\CU_0=\CG\cup\{0\}$.  Let $\CU_{n+1}$ consist of
 $\bigcup_{k\in\Z}\Sg^k\CU_n$, together with a choice of cofibre for
 every morphism in $\CU_n$, and a choice of splitting for every
 idempotent morphism in $\CU_n$.  Put $\CU_\infty=\bigcup_n\CU_n$.  We
 then find that $\CU_\infty$ has only a set of objects, and contains a
 representative of every isomorphism class in $\CU$.
\end{proof}

\begin{proposition}\lbl{prop-preserves-small}
 $\CC_c(P)$ is a thick subderivator of $\CC$ (and so is the same as
 $\gm\CC_c(e)$). 
\end{proposition}
\begin{proof}
 Put $\CE=\gm\CC_c(e)$, which is a thick subderivator; it will suffice
 to show that this is the same as $\CC_c$.
 
 If $F$ is left adjoint to $G$ and $G$ preserves coproducts then for
 small $X$ we have
 \[ [FX,\bigoplus_\al Y_\al] =
    [X,G\bigoplus_\al Y_\al] =
    [X,\bigoplus_\al GY_\al] =
    \bigoplus_\al[X,GY_\al] =
    \bigoplus_\al [FX,Y_\al],
 \]
 so $FX$ is small.  Using this, we see that $u_!$ and $u^*$ preserve
 $\CC_c$.  The claim now follows from Lemma~\ref{lem-gm-left}.
\end{proof}

\begin{corollary}\lbl{cor-comp-gen}
 If $\CC(e)$ is compactly generated, then $\CC(P)$ is compactly
 generated for all $P$.
\end{corollary}
\begin{proof}
 First, we can choose a small skeleton $\CG$ for $\CC_c(e)$, and then
 put
 \[ \CG(P)=\{(i_p)_!(X)\st p\in P,\;X\in\CG\}\sse\CC_c(P). \]
 From the description $\CC_c(P)=(\gm_1\CC_c(e))(P)$ we see that
 $\CC_c(P)$ is generated by $\CG(P)$, and so is essentially small by  
 Lemma~\ref{lem-ess-small}.
 
 Now let $\CT$ be a localising subcategory of $\CC(P)$ that contains
 $\CC_c(P)$.  Put
 \[ \CU =
     \{X\in\CC(e)\st
        (i_p)_!(X) \in \CT \text{ for all } p\in P\}.
 \]
 This is easily seen to be a localising subcategory of $\CC(e)$
 containing $\CC_c(e)$, so $\CU=\CC(e)$.  From
 Theorem~\ref{thm-thick-subcats} it follows that $\gm\CU=\CC$, so in
 particular $\CC(P)=(\gm_1\CU)(P)$.  However, from the definition of
 $\CU$ it is clear that $(\gm_1\CU)(P)\sse\CT$, so $\CT=\CC(P)$ as
 required.  
\end{proof}

\begin{definition}\lbl{defn-comp-gen}
 We say that $\CC$ is \emph{compactly generated} if it satisfies the
 equivalent conditions of Corollary~\ref{cor-comp-gen}.
\end{definition}

\begin{definition}\lbl{defn-perp}
 Let $\CU$ be a thick subcategory of a triangulated category
 $\CT$.  We then write
 \begin{align*}
  \CU^\perp   &= \{Y\in\CT\st [U,Y]=0 \text{ for all } U\in\CU\} \\
  {}^\perp\CU &= \{X\in\CT\st [X,U]=0 \text{ for all } U\in\CU\}.
 \end{align*}
 Similarly, if $\CE$ is a thick subderivator of a stable derivator
 $\CC$, we put $\CE^\perp(P)=\CE(P)^\perp$ and
 $({}^\perp\CE)(P)={}^\perp(\CE(P))$. 
\end{definition}

\begin{proposition}\lbl{prop-perp}
 $\CE^\perp$ and ${}^\perp\CE$ are thick subderivators, so
 $\CE^\perp=\gm(\CE(e)^\perp)$ and ${}^\perp\CE=\gm({}^\perp\CE(e))$. 
\end{proposition}
\begin{proof}
 Suppose that $X\in({}^\perp\CE)(P)$ and $u\:P\to Q$.  For
 $V\in\CE(Q)$ we have $u^*(V)\in\CE(P)$ and so
 $[u_!(X),V]=[X,u^*(V)]=0$.  From this we see that $u_!$ preserves
 ${}^\perp\CE$.  As $u^*$ is left adjoint to $u_*$ and $u_*$ preserves
 $\CE$, we see in the same way that $u^*$ preserves ${}^\perp\CE$.  It
 therefore follows from Lemma~\ref{lem-gm-left} that ${}^\perp\CE$ is
 a thick subderivator.  A dual argument shows that $\CE^\perp$ is also
 a thick subderivator. 
\end{proof}

\section{Anafunctors for derivators}
\label{sec-ana}

Suppose we have morphisms of derivators
\[ \CC\xla{F}\CX\xra{G}\CD \] in which $F$ is an equivalence (which
means that $F\:\CC(P)\to\CX(P)$ is full, faithful and essentially
surjective for all $P$).  We could choose an inverse for $F$ and thus
obtain a morphism $GF^{-1}\:\CC\to\CD$.  However, we prefer not to
make arbitrary choices, so we will instead treat $GF^{-1}$ as a formal
fraction, or \emph{anafunctor}.  The bicategory $\DER'$ of derivators
and anafunctors is thus obtained from the bicategory $\DER$ of
derivators by inverting the equivalences.  A formal framework for
bicategories of fractions is given
in~\cite{pronk:bicategoriesfractions}*{Section 2}, and extended
in~\cite{vitale:bipullbacksandfractions}.  To apply this framework to
derivators, we need the following result:

\begin{lemma}\lbl{lem-two-pullbacks}
 The $2$-category $\DER$ admits $2$-pullbacks along equivalences.
\end{lemma}
\begin{proof}
 Consider a span $\CD\xra{F}\CF\xla{G}\CE$, in which $G$ is an
 equivalence of derivators.  We will provide an explicit model for the
 $2$-pullback.  For each finite poset $P$ we have a span of categories
 $\CD(P)\xra{F_P}\CF\xla{G_P}\CE(P)$, in which $G_P$ is an
 equivalence.  We define $\CP(P)$ to be the usual $2$-pullback of this
 span.  Explicitly, the objects are triples
 $(X,Y,\alpha)$ where $X\in\CD(P)$, $Y\in\CE(P)$ and
 $\alpha\:F_P(X)\to G_P(Y)$ is an isomorphism in $\CF(P)$.  A
 morphism $(f,g)\:(X,Y,\alpha)\to (X',Y',\al')$ consists of a
 pair of morphisms $f\:X\to X'$ and $g\:Y\to Y'$ respectively in
 $\CD(P)$ and $\CE(P)$ such that the diagram
 \begin{center}
  \begin{tikzcd}
   F_P X  \arrow[d, "\alpha" left] \arrow[r, "F_A f"] &
   F_P X' \arrow[d, "\alpha'" right]  \\
   G_P Y  \arrow[r, "G_P g" above] &
   G_P Y'
  \end{tikzcd}
 \end{center}
 commutes.  Now suppose we have a monotone map $u\:P\to Q$.  We define
 $u^*\:\CP(Q)\to\CP(P)$ as follows.  On objects we set
 $u^*(X,Y,\alpha)=(u^*X,u^*Y,\alpha_{u^*})$ where the isomorphism
 $\alpha_{u^*}$ is the unique one making the following square commute
 \begin{center}
  \begin{tikzcd}
   F_P u^*X \arrow[r, "\alpha_{u^*}"] &
   G_P u^* Y   \\
   u^*F_Q X \arrow[u, "\gamma^F_u" left] \arrow[r, "u^*\alpha" above] &
   u^*G_Q Y \arrow[u, "\gamma^G_u" right].
  \end{tikzcd}
 \end{center}
 On morphisms we set $u^*(f,g)=(u^*f, u^*g)$.  It is immediate that
 this is strictly functorial in $u$, so we have defined a
 prederivator.  There are projections $\pi_1\:\CP\to\CD$ and
 $\pi_2\:\CP\to\CE$, and we can define an invertible modification
 $\varphi\:F\pi_1\to G\pi_2$ by $\varphi_{(X,Y,\al)}=\al$.  We thus
 have a diagram as follows:
 \begin{center}
  \begin{tikzcd}
   \CP \arrow[r, "P_2"] \arrow[d, "P_1" left] &
   \CE \arrow[d, "G" right] \arrow[dl, Leftarrow, "\varphi" above left]  \\
   \CD \arrow[r, "F" below] &
   \CF
  \end{tikzcd}
 \end{center}
 That this square provides a model for the $2$-pullback of
 prederivators is a bothersome computation that we leave to the
 reader. 

 From the fact that $G$ is an equivalence, it follows in a standard
 way that $P_1$ is an equivalence.  As $\CD$ is a derivator, it
 follows that $\CP$ is a derivator.  Thus, we have a $2$-pullback in
 the bicategory of derivators.
\end{proof}

We can now argue as in
\cite{vitale:bipullbacksandfractions}*{Prop. 2.8} to justify the
existence of the right calculus of fractions with respect to the class
of equivalences in $\DER$. The preferred squares we use for the
composition of fractions are the explicit $2$-pullbacks constructed
above.

\begin{remark}\lbl{rem-ana}
 We will need two key features of the resulting bicategory, as
 follows.  Firstly, suppose we have a diagram of derivators as follows,
 which commutes on the nose:
 \begin{center}
  \begin{tikzcd}[sep=large]
   \CC &
   \CX
    \arrow[l,"F"',"\simeq"]
    \arrow[dl,"J"']
    \arrow[d,"G"] \\
   \CY
    \arrow[u,"H","\simeq"']
    \arrow[r,"K"'] &
   \CD
  \end{tikzcd}
 \end{center}
 Here it is assumed that $F$ and $H$ are equivalences, and it follows
 that $J$ is also an equivalence.  The diagram then gives rise to an
 isomorphism between the anafunctors $GF^{-1}$ and $KH^{-1}$.

 Next, suppose we fix an equivalence $F\:\CX\to\CC$, and consider
 various different functors $G_i\:\CX\to\CD$.  Then any natural
 transformation $\al\:G_0\to G_1$ gives rise to a morphism
 $G_0F^{-1}\to G_1F^{-1}$ of anafunctors, and this is functorial in
 $\al$.  
\end{remark}

\section{The anafunctors \texorpdfstring{$\phi_A$}{phi A}}
\label{sec-phi}

We now explain our preferred framework for Bousfield localisation in
the context of derivators.  This will rely on facts about Bousfield
localisation in compactly generated triangulated categories.  For the
homotopy category of spectra, all statements are well-known with very
classical proofs that rely on having an underlying geometric category
of spectra.  There are also proofs of similar results in more
axiomatic frameworks, relying only on the theory of triangulated
categories.  These are typically formulated in the context of
well-generated categories as defined by Neeman~\cite{ne:tc}, and the
proofs are somewhat complex.  It is well-known to experts that
everything becomes much simpler, and much closer to the original
results for the category of spectra, if we restrict attention to
compactly generated categories.  However, it seems surprisingly hard
to find an full account of this in the literature. We have therefore
provided one in Appendix~\ref{apx-comp-gen}.

\begin{definition}\lbl{defn-deriv-loc}
 Let $\CC$ be a compactly generated stable derivator, and let
 $K\:\CC(e)\to\Ab$ be a homology theory.  As usual, we extend this to
 a graded theory $K_*\:\CC(e)\to\Ab_*$ by $K_n(X)=K(\Sg^{-n}X)$.  For
 $X\in\CC(P)$ define $K^P(X)=\bigoplus_{p\in P}K(X_p)$, and note that
 this is again a homology theory.  Using
 Theorem~\ref{thm-thick-subcats} we see that the subcategories
 $\ker(K_*^P)\sse\CC(P)$ form a localising subderivator of $\CC$,
 which we will just call $\ker(K_*)$.

 Now suppose we have an object $X\in\CC([1]\tm P)$.  This gives a
 morphism $u\:X_0\to X_1$ in $\CC(P)$ in the usual way.  We say that
 $X$ is a \emph{localisation object} if $\fib(u)\in\ker(K^P_*)$ and
 $X_1\in\ker(K^P_*)^\perp$.  We write $\CL(P)$ for the subcategory of
 localisation objects in $\CC([1]\tm P)$.  This is clearly a
 subprederivator of $\CC^{[1]}$, and we have a morphism
 $i_0^*\:\CL\to\CC$ of prederivators. 
\end{definition}

\begin{proposition}\lbl{prop-deriv-loc}
 $\CL$ is a thick subderivator of $\CC^{[1]}$, and $i_0^*\:\CL\to\CC$
 is an equivalence of derivators.
\end{proposition}
\begin{proof}
 We know from Theorem~\ref{thm-thick-subcats} and
 Proposition~\ref{prop-perp} that $\ker(K_*)$ and $\ker(K_*)^\perp$ are
 thick subderivators.  Together with the results that we recalled in
 Theorem~\ref{thm-triangulation}, this implies that $\CL$ is a thick
 subderivator of $\CC^{[1]}$.  Now consider the functor
 $i_0^*\:\CL(P)\to\CC(P)$.  Given $X\in\CC(P)$, we can find a
 distinguished triangle $CX\to X\to LX$ with $CX\in\ker(K^P_*)$ and
 $LX\in\ker(K^P_*)^\perp$, by Theorem~\ref{thm-loc} and
 Proposition~\ref{prop-homology-weakly-initial}.  By the strongness
 condition (Definition~\ref{defn-stable-derivator}(a)), we can find
 $Y\in\CC([1]\tm P)$ such that the resulting morphism $Y_0\to Y_1$ is
 isomorphic to $X\to LX$.  This proves that $i_0^*$ is essentially
 surjective.  Now suppose we have $X,Y\in\CL(P)$.
 Proposition~\ref{prop-recollement} gives us a distinguished triangle
 $i_{1!}i_1^*(Y)\to Y\to i_{0*}i_0^*(Y)$.  By applying $[X,-]$ to this
 and using various adjunctions we obtain an exact sequence
 \[ [i_1^?X,i_1^*Y] \to [X,Y] \xra{i_0^*} [i_0^*X,i_0^*Y] \to 
    [i_1^?X,i_1^*\Sg Y].
 \]
 Kere $i_1^?X$ is just the cofibre of $i_0^*X\to i_1^*X$, which lies
 in $\ker(K^P_*)$ because $X\in\CL(P)$.  Also, because $Y\in\CL(P)$ we
 have $i_1^*Y\in\ker(K^P_*)^\perp$, so
 $[i_1^?X,i_1^*Y]=[i_1^?X,i_1^*\Sg Y]=0$.  It follows that the map
 $i_0^*\:[X,Y]\to[i_0^*X,i_0^*Y]$ is an isomorphism as required.
\end{proof}

We will sometimes need a slightly more general statement.
\begin{lemma}\lbl{lem-deriv-loc-alt}
 Suppose that $X,Y\in\CC^{[1]}(P)$, giving maps $u\:X_0\to X_1$ and
 $v\:Y_0\to Y_1$ in $\CC(P)$.  Suppose that $\fib(u)\in\ker(K^P_*)$ and
 $Y_1\in\ker(K^P_*)^\perp$.  Then the map
 \[ i_0^* \: \CC([1]\tm P)(X,Y) \to \CC(P)(X_0,Y_0) \]
 is bijective.
\end{lemma}
\begin{proof}
 These weakened hypotheses are all that was used in the proof of
 Proposition~\ref{prop-deriv-loc} 
\end{proof}

\begin{definition}\lbl{defn-LK}
 We define $L_K$ to be the anafunctor $i_1^*(i_0^*)^{-1}$, and call
 this \emph{Bousfield localisation} with respect to $K$.
\end{definition}

We now return to a framework similar to that of
Section~\ref{sec-basic}: we assume that we have a compactly generated
stable derivator $\CC$ together with homology theories $K(i)$ on
$\CC(e)$ for $i\in N$, where $N$ is a finite, totally ordered set.  As
before, we define $\PP$ to be the poset of subsets of $N$ (ordered by
inclusion).  We also define $\QQ$ to be the poset of upper sets in
$\PP$ (ordered by reverse inclusion), and define $u\:\PP\to\QQ$ by
$uT=\{U\st T\sse U\}$.

\begin{definition}\lbl{defn-fully-loc}
 Consider a finite poset $R$ and an object $X\in\CC(\PP\tm R)$.  
 Suppose that $t\in N$ and $U\sse N$ with $t<u$ for all $u\in U$.  We
 then write $tU$ for $\{t\}\cup U$, so we have $U<tU$ in $\PP$, giving
 maps $f_{t,U,r}\:X_{U,r}\to X_{tU,r}$ in $\CC(e)$.  We say that $X$ is
 \emph{$(t,U)$-localising} if $f_{t,U}$ is a $K(t)$-localisation.
 Equivalently, $X_{tU,r}$ should be $K(t)$-local, and the fibre of
 $f_{t,U,r}$ should be $K(t)$-acyclic.  We also say that $X$ is
 \emph{fully localising} if it is $(t,U)$-localising for all $t$ and
 $U$.  We write $\CP(R)$ for the full subcategory of fully localising
 objects in $\CC(\PP\tm R)$.  There is an evident inclusion
 $i_\emptyset\:R\to\PP\tm R$, which gives a functor
 $i_\emptyset^*\:\CP(R)\to\CC(R)$.
\end{definition}

\begin{example}\lbl{eg-fract-cube}
 Consider the case $N=\{0\}$, and suppose that our derivator $\CC$
 arises from a stable model category $\CC_0$.  An object of $\CP(R)$
 is then a diagram $X\:[1]\tm R\to\CC_0$ such that the morphisms
 $X_{0r}\to X_{1r}$ are all localisations with respect to $K(0)$.
 Informally, we can therefore say that an object of $\CP$ is a diagram
 of type $(X\to L_{K(0)}X)$.  In the same sense, if $N=\{0,1\}$ then
 an object of $\CP$ is essentially a diagram of the following type:
 \begin{center}
  \begin{tikzcd}
   X \arrow[r] \arrow[d] & L_{K(0)}X \arrow[d] \\
   L_{K(1)}X \arrow[r] & L_{K(0)}L_{K(1)}X 
  \end{tikzcd}
  \hspace{3em}
  \begin{tikzcd}
   X \arrow[r] \arrow[d] & \phi_0X \arrow[d] \\
   \phi_1X \arrow[r] & \phi_{01}X 
  \end{tikzcd}
 \end{center}
 Here the right hand diagram is just alternate notation for the left
 hand one.  For $N=\{0,1,2\}$, the diagram is as follows:
 \begin{center}
  \begin{tikzpicture}[scale=0.60]
    \node (12) at (2,0) {$\phi_{12}X$};
    \node (012) at (6,0) {$\phi_{012}X$};
    \node (2) at (0,2) {$\phi_2X$};
    \node (02) at (4,2) {$\phi_{02}X$};
          \draw[->] (12) -- (012);
          \draw[->] (02) -- (012);
          \draw[->] (2) -- (12);
          \draw[->] (2) -- (02);
    \node (1) at (2,3) {$\phi_1X$};
    \node (01) at (6,3) {$\phi_{01}X$};
    \node (empty) at (0,5) {$X$};
    \node (0) at (4,5) {$\phi_0X$};
          \draw[->] (1) -- (01);
          \draw[->] (0) -- (01);
          \draw[->] (empty) -- (1);
          \draw[->] (empty) -- (0);
          \draw[->] (empty) -- (2);
          \draw[->] (0) -- (02);
          \draw[->] (1) -- (12);
          \draw[->] (01) -- (012);
  \end{tikzpicture}
 \end{center}       
\end{example}

\begin{proposition}\lbl{prop-fully-loc}
 $\CP$ is a thick subderivator of $\CC^{\PP}$, and
 $i_\emptyset^*\:\CP\to\CC$ is an equivalence of derivators.
\end{proposition}
\begin{proof}
 The claim is clear if $N=\emptyset$.  If $N\neq\emptyset$, we let
 $n_0\in N$ be the smallest element, so $N$ can be decomposed as
 $\{n_0\}\amalg N_1$ say.  This gives an obvious decomposition
 $\PP=[1]\tm\PP_1$.  We can define $\CP_1\sse\CC^{\PP_1}$ using $N_1$,
 and by induction we can assume that this is a thick subderivator with
 $i_\emptyset^*\:\CP_1\to\CC$ being an equivalence.
 
 Now define $\CL$ as in Definition~\ref{defn-deriv-loc}, with
 respect to the homology theory $K(n_0)$ for the derivator $\CP_1$.
 We find that 
 \[ \CP(R) = \{X\in \CL(\PP_1\tm R) \st i_0^*(X)\in \CP_1(R)\}, \]
 and the claim now follows from Proposition~\ref{prop-deriv-loc}
 together with the induction hypothesis.
\end{proof}

Again, we will sometimes need a slightly more general statement.
\begin{lemma}\lbl{lem-fully-loc-alt}
 Suppose that $X,Y\in\CC^{\PP}(R)$.  Suppose that for $t$, $U$ and $r$
 as before, the map $f_{t,U,r}\:X_{U}\to X_{tU}$ is a
 $K(t)$-equivalence, and the object $Y_{tU}$ is $K(t)$-local.  Then
 the map 
 \[ i_\emptyset^* \: \CC(\PP\tm R)(X,Y) \to
      \CC(R)(X_\emptyset,Y_\emptyset) \]
 is bijective.
\end{lemma}
\begin{proof}
 As above taking $n_0$ the minimum of $N$ we get a decomposition
 $\PP=[1]\times \PP_1$ and arguing by induction we can assume the map
 $\CC(\PP_1\tm R)(X_0,Y_0) \to \CC(R)(X_\emptyset,Y_\emptyset)$ is a
 bijection. Now we have just to compose this with
 $\CC(\PP\tm R)(X,Y) \to \CC(\PP_1 \times R)(X_0,Y_0)$ which is an
 isomorphism by Lemma~\ref{lem-deriv-loc-alt}.
\end{proof}

\begin{definition}\lbl{defn-phi-A}
 For $A\sse N$ we consider the diagram
 \[ \CC \xla{i_\emptyset^*} \CP \xrightarrow{i_A^*}\CC \]
 and define an anafunctor $\phi_A\:\CC\to\CC$ by
 $\phi_A=i_A^*\circ(i_\emptyset^*)^{-1}$.
\end{definition}

\begin{remark}\lbl{rem-phi-A}
 Suppose that $A=\{a_1,\dotsc,a_r\}$ with $a_1<\dotsb<a_r$.  It is
 then not hard to see from the definitions that in some sense we have 
 \[ \phi_A \simeq L_{K(a_1)} \dotsb L_{K(a_r)}, \]
 so that Definition~\ref{defn-phi-A} is a more precise version of
 Definition~\ref{defn-P} in the introduction.  A rigorous formulation
 with anafunctors will be given in Corollary~\ref{cor-phi-factor}.
\end{remark}

\begin{proposition}\lbl{prop-phi-single}
 The anafunctor $\phi_{\{a\}}$ is equivalent to $L_{K(a)}$.
\end{proposition}
\begin{proof}
 Define $j\:[1]\to\PP$ by $j(0)=\emptyset$ and $j(1)=\{a\}$.  This
 gives a morphism $j^*\:\CC^{\PP}\to\CC^{[1]}$.  If we define $\CL$ as
 in Definition~\ref{defn-deriv-loc} with respect to $K(a)$, we find
 that $j^*$ restricts to give a morphism $\CP\to\CL$.  This fits into
 a diagram as follows, which commutes on the nose:
 \begin{center}
  \begin{tikzcd}[sep=huge]
   \CP
    \arrow[r,"i^*_{\{a\}}"]
    \arrow[dr,"j^*","\simeq"']
    \arrow[d,"i^*_\emptyset"',"\simeq"] &
   \CC \\
   \CC &
   \CL
    \arrow[l,"i_0^*","\simeq"']
    \arrow[u,"i_1^*"]
  \end{tikzcd}
 \end{center}
 From this it is clear that the anafunctor
 $\phi_{\{a\}}=i_{\{a\}}^*(i_\emptyset^*)^{-1}$ is equivalent to
 $L_{K(a)}=i_1^*(i_0^*)^{-1}$. 
\end{proof}

Now let $j$ be the inclusion of $\PP'=\PP\sm\{\emptyset\}$ in $\PP$,
and consider the fibration
\[ \tfib(X) \xra{} i_\emptyset^*(X) \to \holim_{\PP'}j^*(X) \]
as in Lemma~\ref{lem-tfib-fibre}.

We can now give a derivator formulation of the chromatic fracture
argument.
\begin{proposition}\lbl{prop-fracture}
 For any $X\in\CP(R)$, the above morphism
 $i_\emptyset^*(X)\to\holim_{\PP'}j^*(X)$ is a localisation with
 respect to $K(N)=\bigoplus_{n\in N}K(n)$.
\end{proposition}
\begin{proof}
 From the definition of $\CP(R)$ we see that for all $T\in\PP'$, the
 object $j^*(X)_T$ is local with respect to $K(\min(T))$ and thus with
 respect to $K(N)$.  Proposition~\ref{prop-perp} tells us that the
 $K(N)$-local objects form a thick subderivator, so the object
 $LX=\holim_{\PP'}j^*(X)$ is $K(N)$-local.  It will therefore suffice
 to show that the fibre $\tfib(X)=\text{fib}(X_\emptyset\to LX)$ is
 $K(N)$-acyclic, or equivalently, that it is $K(i)$-acyclic for all
 $i$.  Let $Y_r$ be the total fibre of the subdiagram indexed by
 subsets of $\{r+1,\dotsc,n^*-1\}$, so $Y_{n^*-1}=0$ and
 $Y_{-1}=\tfib(X)$.  Let $Z_r$ be the total fibre of the subdiagram
 indexed by subsets of $\{r,\dotsc,n^*-1\}$ containing $r$, so that
 $Y_r\to Z_r$ is a $K(r)$-localisation, and the fibre is $Y_{r-1}$ by
 Proposition~\ref{prop-tfib-split}.  This shows that $Y_{i-1}$ is
 $K(i)$-acyclic.  The fracture axiom then tells us that $Z_{i-1}$ is
 also $K(i)$-acyclic, and $Y_{i-2}$ is the fibre of the map
 $Y_{i-1}\to Z_{i-1}$ so it is again $K(i)$-acyclic.  By iterating
 this, we find that $Y_{-1}$ is $K(i)$-acyclic as required.
\end{proof}

\begin{definition}\lbl{defn-doubly-localising}
 We say that an object $X\in\CC(\PP\tm\PP\tm R)$ is \emph{doubly
  localising} if
 \begin{itemize}
  \item[(a)] $X$ is fully localising relative to $\PP\tm R$, so it
   lies in $\CP(\PP\tm R)$.
  \item[(b)] The restriction to
   $\{u\emptyset\}\tm\PP\tm R\simeq\PP\tm R$ lies in $\CP(R)$.
 \end{itemize}
 We write $\CP_2(R)$ for the full subcategory of doubly localising
 objects in $\CC(\PP\tm\PP\tm R)$.
\end{definition}

\begin{proposition}\lbl{prop-doubly-loc}
 $\CP_2$ is a thick subderivator of $\CC^{\PP \times \PP}$, and
 $i_{(\emptyset, \emptyset)}^*\colon \CP_2 \rightarrow \CC$ is an
 equivalence.
\end{proposition}
\begin{proof}
 The first claim follows from Theorem~\ref{thm-thick-subcats}. The two
 properties of the previous definition easily imply that the
 equivalence $i_\emptyset^*\colon \CP \rightarrow \CC$ of
 Proposition~\ref{prop-fully-loc} (shifted by a $\PP$ component)
 restricts to an equivalence $\CP_2 \rightarrow \CP$, compose this
 again with $i_\emptyset^*$ and the second claim is verified.
\end{proof}

It is easy to see that $X\in\CC(\PP\tm\PP\tm R)=\CC^R(\PP\tm\PP)$ is
doubly localising iff the following conditions are satisfied:
\begin{itemize}
 \item[(a)] For all $a,A,B$ with $\{a\}\angle A$, the map
  $X_{A,B}\to X_{\{a\}\cup A,B}$ is a $K(a)$-localisation.
 \item[(b)] For all $b,B$ with $\{b\}\angle B$, the map
  $X_{\emptyset,B}\to X_{\emptyset,\{b\}\cup B}$ is a
  $K(b)$-localisation. 
\end{itemize}
This essentially means that if $A=\{a_1<\dotsb<a_p\}$ and
$B=\{b_1<\dotsb<b_q\}$ we must have 
\[ X_{A,B} =
     L_{K(a_1)} \dotsb L_{K(a_p)} L_{K(b_1)} \dotsb L_{K(b_q)}
      X_{(\emptyset,\emptyset)}.
\]
In particular, we see that $X_{(A,B)}=0$ unless $A\angle B$.  This
motivates the following construction.

\begin{definition}\lbl{defn-doubly-localising-alt}
 We put $\MM=\{(A,B)\in\PP\tm\PP \st A\angle B\}$, and define
 $\sg\:\MM\to\PP$ by $\sg(A,B)=A\cup B$.  We say that an
 object $X\in\CC(\MM\tm R)=\CC^R(\MM)$ is \emph{doubly localising} if
 \begin{itemize}
 \item[(a)] For all $a,A,B$ with $\{a\}\angle A$ and
   $\{a\}\cup A\angle B$, the map
   $X_{A,B}\to X_{\{a\}\cup A,B}$ is a $K(a)$-localisation.
  \item[(b)] For all $b,B$ with $\{b\}\angle B$, the map
   $X_{\emptyset,B}\to X_{\emptyset,\{b\}\cup B}$ is a
   $K(b)$-localisation. 
 \end{itemize}
 We write $\CP'_2(R)$ for the subcategory of doubly localising objects.
\end{definition}

\begin{proposition}\lbl{prop-doubly-loc-alt}
 $\CP'_2$ is a thick subderivator of $\CC^{\MM}$, and the inclusion
 $\inc\:\MM\to\PP\tm\PP$ induces mutually inverse equivalences
 \[ \CP'_2 \xra{\inc_*} \CP_2 \xra{\inc^*} \CP'_2, \]
 and the morphism $i_{(\emptyset,\emptyset)}^*\:\CP'_2\to\CC$ is also
 an equivalence.  Moreover, the map $\sg\:\MM\to\PP$ gives an
 equivalence $\sg^*\:\CP\to\CP'_2$ and thus an equivalence
 $\inc_*\circ\sg^*\:\CP\to\CP_2$. 
\end{proposition}
\begin{proof}
 The subposet $\MM\sse\PP\tm\PP$ is a sieve, so
 Proposition~\ref{prop-sieve} gives mutually inverse equivalences
 \[ \CC(\MM\tm R) \xra{\inc_*} \CC_{\MM\tm R}(\PP\tm\PP\tm R)
      \xra{\inc^*} \CC(\MM\tm R).
 \]
 We have observed that if $X$ is doubly localising then $X_{A,B}=0$
 for $(A,B)\not\in\MM$, so $X\in\CC_{\MM\tm R}(\PP\tm\PP\tm R)$.  It
 follows that $\inc^*$ restricts to give an equivalence from
 $\CP_2(R)$ to some subcategory of $\CC(\MM\tm R)$, with inverse given
 by $\inc_*$.  It is easy to check that the relevant subcategory is
 $\CP'_2(R)$.  We have now seen that in the diagram 
 \[ \CP_2 \xra{\inc^*} \CP'_2
  \xra{i_{(\emptyset,\emptyset)}^*} \CC,
 \]
 the first map and the composite are both equivalences of
 (pre)derivators, so the second map is also an equivalence.  From this
 it also follows that $\CP'_2$ is a derivator.  Finally, direct
 inspection of the definitions shows that $\sg^*(\CP)\sse\CP'_2$, and
 $i_{(\emptyset,\emptyset)}^*\sg^*=i_{\emptyset}^*$, which implies
 that $\sg^*\:\CP\to\CP'_2$ is an equivalence.
\end{proof}

\begin{proposition}\lbl{prop-phi-comp}
 If $A\angle B$ then there is an equivalence of anafunctors
 $\phi_A\phi_B\simeq\phi_{A\cup B}$.  
\end{proposition}
\begin{proof}
 Define $j_B\:\PP\to\PP\tm\PP$ by $j_B(T)=(T,B)$.  Consider the
 following diagram:  
 
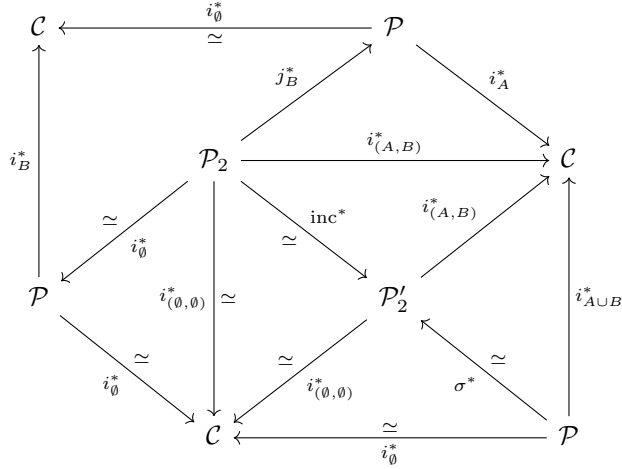
\begin{center}
  \begin{tikzcd}[sep=huge]
   \CC & &
   \CP
    \arrow[dr,"i_A^*"]
    \arrow[ll,"i_\emptyset^*"',"\simeq"] \\ &
   \CP_2
    \arrow[rr,"i_{(A,B)}^*"]
    \arrow[dr,"\inc^*","\simeq"']
    \arrow[dd,"i_{(\emptyset,\emptyset)}^*"',"\simeq"]
    \arrow[ur,"j_B^*"]
    \arrow[dl,"i_\emptyset^*","\simeq"'] & &
   \CC \\
   \CP
    \arrow[uu,"i_B^*"]
    \arrow[dr,"i_\emptyset^*"',"\simeq"] & &
   \CP'_2
    \arrow[dl,"i_{(\emptyset,\emptyset)}^*","\simeq"']
    \arrow[ur,"i_{(A,B)}^*"] \\ &
   \CC & &
   \CP
    \arrow[ll,"i_\emptyset^*","\simeq"']
    \arrow[uu,"i_{A\cup B}^*"']
    \arrow[ul,"\sg^*","\simeq"']
  \end{tikzcd}
 \end{center}
 Given that $\CC$ is a strict $2$-functor, we see that everything
 commutes on the nose.  Several morphisms have been marked as
 equivalences; these are justified by
 Proposition~\ref{prop-doubly-loc-alt}.  It follows that all routes
 from the middle bottom $\CC$ to the middle right $\CC$ give the same
 anafunctor up to equivalence.  If we go clockwise around the edge of
 the diagram we get $\phi_A\phi_B$, and if we go anticlockwise we get
 $\phi_{A\cup B}$.
 
 To be precise the composition of anafunctors $\phi_A \phi_B$ is given
 by the following composition of spans via pullback
 
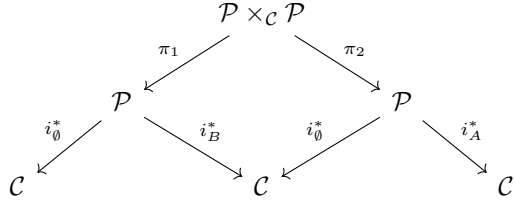
\begin{center}
  \begin{tikzcd}
   & & \arrow[dl, "\pi_1" above left]
   \CP\times_{\CC} \CP \arrow[dr, "\pi_2"] & & \\
   & \arrow[dl, "i_\emptyset^*" above left]
   \CP \arrow[dr, "i_B^*"] & &
   \arrow[dl, "i_\emptyset^*" above left] \CP \arrow[dr, "i_A^*"] & \\
   \CC & &\CC & & \CC
  \end{tikzcd}
 \end{center} 
 The upper left part of the previous diagram means that we can easily
 produce an isomorphism of anafunctors 
 \begin{center}
  \begin{tikzcd}
   & \arrow[dl, "i_\emptyset^* \pi_1" above left ]
   \CP\times_{\CC} \CP \arrow[dr, "i_A^* \pi_2"] &  \\
  \CC &  & \CC \\
  & \arrow[ul, " i_{(\emptyset, \emptyset )}^*"]
  \CP_2 \arrow[uu, "\psi"]  \arrow[ur, "i_{(A,B)}^*" below right] &
  \end{tikzcd}
 \end{center}  
 where $\psi$ is obtained using $i_\emptyset^*$ and $j_B^*$.
\end{proof}

\begin{corollary}\lbl{cor-phi-factor}
 Suppose that $A=\{a_1,\dotsc,a_r\}$ with $a_1<\dotsb<a_r$.  There is
 then an equivalence of anafunctors 
 \[ \phi_A \simeq L_{K(a_1)} \dotsb L_{K(a_r)}. \]
\end{corollary}
\begin{proof}
 This follows by induction using Propositions~\ref{prop-phi-single}
 and~\ref{prop-phi-comp}.  (The base case $A=\emptyset$ says that
 $\phi_\emptyset$ is equivalent to the identity, which is clear
 because $\phi_\emptyset=i_\emptyset^*(i_\emptyset^*)^{-1}$ by
 definition.) 
\end{proof}

\section{The anafunctors \texorpdfstring{$\tht_U$}{theta U}}
\label{sec-tht}

We now start to define a more general class of iterated localisation
functors. 

\begin{definition}\lbl{defn-fracture-obj}
 Consider an object $X\in\CC(\QQ\tm R)$, and the pullback
 $(u\tm 1)^*(X)\in\CC(\PP\tm R)$.  We say that $X$ is
 \emph{$u$-cartesian} if the natural map
 \[ (u\tm 1)^* \: \CC(\QQ\tm R)(W,X) \to
    \CC(\PP\tm R)((u\tm 1)^*(W),(u\tm 1)^*(X))
 \]
 is an isomorphism for all $W$, or equivalently, $X$ is in the
 essential image of the functor
 \[ (u\tm 1)_* \: \CC(\PP\tm R) \to \CC(\QQ\tm R). \]
 We say that $X$ is a \emph{fracture object} if it is $u$-cartesian,
 and $(u\tm 1)^*(X)$ is fully localising.  We write $\CF(R)$ for the
 subcategory of fracture objects in $\CC(\QQ\tm R)$.  We also define
 $j\:R\to\QQ\tm R$ by $j(r)=(u\emptyset,R)$.
\end{definition}

\begin{remark}\lbl{rem-u-cartesian}
 Because we have ordered $\QQ$ by reverse inclusion, we have
 $U\leq uA$ iff $uA\sse U$ iff $A\in U$.  Using this together with the
 Kan formula, the $u$-cartesian condition becomes
 \[ X_U = \holim_{A\in U} X_{uA}. \]
\end{remark}

\begin{example}\lbl{eg-fracture-obj}
 Consider the case $N=\{0,1\}$, and put
 \[ W = vN = u\{0\} \cup u\{1\} = \{A\sse N\st A\neq\emptyset\}
      \in \QQ.
 \]
 We then have $\QQ = \{uA\st A\in\PP\}\cup\{W,\emptyset\}$, with
 partial order as shown on the left below.  A fracture object is as
 shown on the right.
 
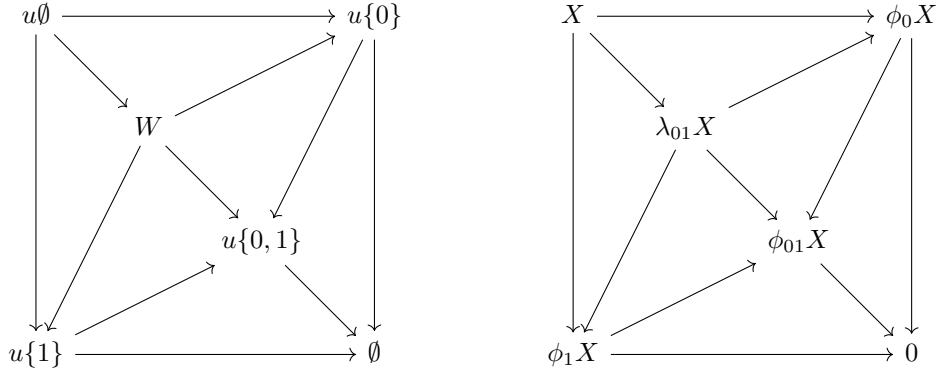
\begin{center}
  \begin{tikzcd}[sep={1.5cm,between origins}]
   u\emptyset \arrow[rrr] \arrow[dr] \arrow[ddd] &&&
   u\{0\} \arrow[ddl] \arrow[ddd] \\ &
   W \arrow[urr] \arrow[dr] \arrow[ddl] \\ &&
   u\{0,1\} \arrow[dr] \\
   u\{1\} \arrow[urr] \arrow[rrr] &&&
   \emptyset
  \end{tikzcd}
  \hspace{4em}
  \begin{tikzcd}[sep={1.5cm,between origins}]
   X \arrow[rrr] \arrow[dr] \arrow[ddd] &&&
   \phi_0X \arrow[ddl] \arrow[ddd] \\ &
   \lm_{01}X \arrow[urr] \arrow[dr] \arrow[ddl] \\ &&
   \phi_{01}X \arrow[dr] \\
   \phi_1X \arrow[urr] \arrow[rrr] &&&
   0
  \end{tikzcd}
 \end{center}
 Now consider the case $N=\{0,1,2\}$, where $|\PP|=8$.  It turns out
 that $|\QQ|=20$, with elements as follows:
 \begin{itemize}
  \item The smallest element is $u\emptyset=\PP$.
  \item The next smallest element is
   $vN=\{A\in\PP\st A\neq\emptyset\}$.
  \item We write $v_{ij}=v\{i,j\}=\{A\st i\in A\text{ or}\ j\in A\}$.
  \item We write $x_i=\{A\st i\in A\text{ or } \{i\}^c\sse A\}$.
  \item We write $u_i=u\{i\}=v\{i\}=\{A\st i\in A\}$.
  \item We write $w_i=\{A\st A\supset\{i\}\}$ (strict inclusion).
  \item We write $u_{ij}=u\{i,j\}=\{A\st \{i,j\}\sse A\}=\{\{i,j\},N\}$.
  \item We write $u_{012}=uN=\{N\}$.
  \item We write $y=\{A\st |A|\geq 2\}$.
  \item The largest element is $\emptyset$.
 \end{itemize}
 The Hasse diagram for $\QQ\sm\{u\emptyset,vN,y,\emptyset\}$ can be
 drawn in the plane as follows:
 
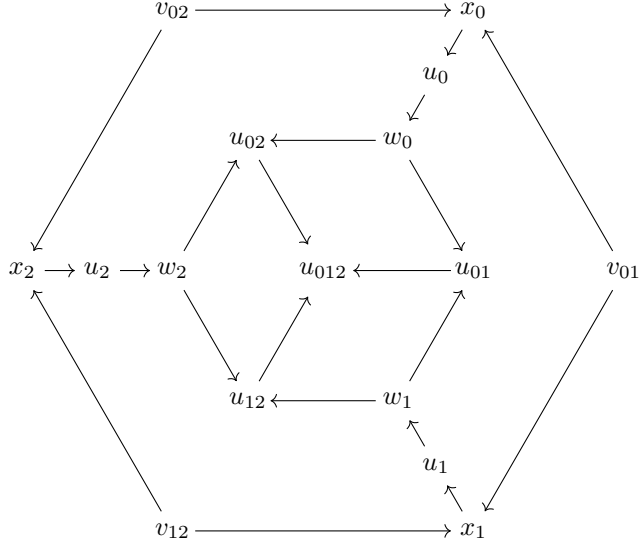
\begin{center}
  \begin{tikzpicture}
   \draw (0,0)     node{$u_{012}$};
   \draw (  0:2.0) node{$u_{01}$}; 
   \draw (120:2.0) node{$u_{02}$}; 
   \draw (240:2.0) node{$u_{12}$}; 
   \draw ( 60:2.0) node{$w_{0}$};
   \draw (180:2.0) node{$w_{2}$};
   \draw (300:2.0) node{$w_{1}$};
   \draw ( 60:3.0) node{$u_{0}$};
   \draw (180:3.0) node{$u_{2}$};
   \draw (300:3.0) node{$u_{1}$};
   \draw ( 60:4.0) node{$x_{0}$};
   \draw (180:4.0) node{$x_{2}$};
   \draw (300:4.0) node{$x_{1}$};
   \draw (  0:4.0) node{$v_{01}$}; 
   \draw (120:4.0) node{$v_{02}$}; 
   \draw (240:4.0) node{$v_{12}$};
   \draw[->] (  0:1.7) -- (  0:0.4); 
   \draw[->] (120:1.7) -- (120:0.4); 
   \draw[->] (240:1.7) -- (240:0.4); 
   \draw[->] ( 60:2.7) -- ( 60:2.3);
   \draw[->] (180:2.7) -- (180:2.3);
   \draw[->] (300:2.7) -- (300:2.3);
   \draw[->] ( 60:3.7) -- ( 60:3.3);
   \draw[->] (180:3.7) -- (180:3.3);
   \draw[->] (300:3.7) -- (300:3.3);
   \draw[shorten >=0.3cm,shorten <=0.3cm,->] ( 60:2.0) -- (  0:2.0);
   \draw[shorten >=0.3cm,shorten <=0.3cm,->] ( 60:2.0) -- (120:2.0);
   \draw[shorten >=0.3cm,shorten <=0.3cm,->] (180:2.0) -- (120:2.0);
   \draw[shorten >=0.3cm,shorten <=0.3cm,->] (180:2.0) -- (240:2.0);
   \draw[shorten >=0.3cm,shorten <=0.3cm,->] (300:2.0) -- (240:2.0);
   \draw[shorten >=0.3cm,shorten <=0.3cm,->] (300:2.0) -- (  0:2.0);
   \draw[shorten >=0.3cm,shorten <=0.3cm,<-] ( 60:4.0) -- (  0:4.0);
   \draw[shorten >=0.3cm,shorten <=0.3cm,<-] ( 60:4.0) -- (120:4.0);
   \draw[shorten >=0.3cm,shorten <=0.3cm,<-] (180:4.0) -- (120:4.0);
   \draw[shorten >=0.3cm,shorten <=0.3cm,<-] (180:4.0) -- (240:4.0);
   \draw[shorten >=0.3cm,shorten <=0.3cm,<-] (300:4.0) -- (240:4.0);
   \draw[shorten >=0.3cm,shorten <=0.3cm,<-] (300:4.0) -- (  0:4.0);
  \end{tikzpicture}
 \end{center}
 The remaining vertices fit in as follows.  At the bottom we have
 $u\emptyset$, which is covered by $vN$, which is covered by the
 elements $v_{ij}$.  At the top, $u_{012}$ is covered by $\emptyset$.
 In the middle, $y$ covers the elements $x_i$ and is covered by the
 elements $w_i$.  The $u$ and $v$ elements will correspond to functors
 $\phi$ and $\lm$ as we have discussed previously.  One can check that
 the remaining elements other than $w_1$ can be factored as follows,
 and so will also correspond to iterated localisation functors:
 \begin{align*}
  x_0 &= v_{01} * v_{02} &
  x_1 &= v_{01} * v_{12} &
  x_2 &= v_{02} * v_{12} \\
  w_0 &= u_0 * v_{12} &
  w_2 &= v_{01} * u_2 &
  y   &= v_{01} * v_{02} * v_{12}. 
 \end{align*}
 However, Example~\ref{eg-not-thread} shows that $w_1$ cannot be
 factored in this way.
\end{example}

\begin{proposition}\lbl{prop-fracture-obj}
 $\CF$ is a thick subderivator of $\CC^\QQ$, and $j^*\:\CF\to\CC$ is
 an equivalence of derivators.
\end{proposition}
\begin{proof}
 Let $\CE(R)$ be the subcategory of $u$-cartesian objects in
 $\CC(\QQ\tm R)$, so the functor $(u\tm 1)^*\:\CE(R)\to\CC(\PP\tm R)$
 is an equivalence.  An object is $u$-cartesian iff the unit map
 $X\to (u\tm 1)_*(u\tm 1)^*(X)$ is an isomorphism, and from this we
 see that $\CE(R)$ is a thick subcategory of $\CC(\QQ\tm R)$.  (Here
 we have used Theorem~\ref{thm-triangulation}, as we will do
 repeatedly without further comment.)  

 Now consider a morphism $f\:R_0\to R_1$ of finite posets.  We know
 from~\cite{groth:derpointstable}*{Proposition 2.6} that the resulting
 derivator morphism $f^*\:\CC^{R_1}\to\CC^{R_0}$ preserves homotopy
 Kan extensions, so it commutes (in an evident sense) with the
 functors $(u\tm 1)_*$, so it restricts to give a functor
 $f^*\:\CE(R_1)\to\CE(R_0)$.  It is even clearer that the functors
 $(1\tm f)_*$ commute with $(u\tm 1)_*$, so they restrict to give
 $f_*\:\CE(R_0)\to\CE(R_1)$.  By the dual of Lemma~\ref{lem-gm-left},
 we deduce that $\CE$ is a thick subderivator of $\CE^\QQ$.

 Next, consider the preimage under $u^*\:\CC^{\QQ}\to\CC^{\PP}$ of the
 thick subderivator $\CP\sse\CC^{\PP}$.
 Using~\cite{groth:derpointstable}*{Proposition 2.6} again, we see
 that this preimage is again a thick subderivator.  By intersecting
 this with $\CE$, we see that $\CF$ is a thick subderivator as
 claimed.  
 
 In conclusion we obtained the diagram of derivators
 \begin{center}
  \begin{tikzcd}
   \CC^{\PP} \arrow[r, "u_*"]  & \CE \\
   \CP \arrow[rightarrowtail, u] \arrow[r, "u_*"'] &
   \CF \arrow[rightarrowtail, u]
  \end{tikzcd}
 \end{center}
 where the vertical arrows are inclusions and the horizontal ones
 equivalences. Thus composing the inverse equivalence
 $u^* \colon \CF \rightarrow \CP$ with $i_\emptyset^*$ of
 Proposition~\ref{prop-fully-loc} we get the second claim. 
\end{proof}

\begin{definition}\lbl{defn-tht}
 For $U\in\QQ$, we define $\tht_U\:\CC\to\CC$ to be the anafunctor
 \[ \CC \xra{(j^*)^{-1}} \CF \sse \CC^{\QQ} \xra{i_U^*} \CC. \]
 We also note that an inequality $U\leq V$ gives a natural
 transformation $i_U^*\to i_V^*$ and thus a morphism $\tht_U\to\tht_V$
 of anafunctors, as discussed in Remark~\ref{rem-ana}.
\end{definition}

\begin{remark}\lbl{rem-tht}
 Consider an object $X\in\CC(R)$, and a fracture object $Y\in\CF(R)$
 with $j^*Y\simeq X$.  Then the object $Y_U=i_U^*Y\in\CC(R)$ is a
 choice of $\tht_U(X)$.  As $Y$ is $u$-cartesian we have
 \[ Y_U = \holim_{U\leq uA} Y_{uA} = \holim_{A\in U} Y_{uA}. \]
 Also, as $u^*Y\in\CP(R)$ we know that $Y_{uA}$ is a choice of
 $\phi_A(X)$.  Thus, the basic idea is that
 $\tht_U(X)=\holim_{A\in U}\phi_A(X)$.
\end{remark}

\begin{remark}\lbl{rem-kp-morava}
 Consider the original chromatic context where the homology theory
 $K(n)$ is represented by a spectrum, so we can apply $\phi_A$ or
 $\tht_U$ to that spectrum.  It is easy to see that $\phi_A(K(n))$ is 
 $K(n)$ if $A\sse\{n\}$, and $\phi_A(K(n))=0$ in all other cases.
 From this we find that $\tht_U(K(n))$ is $K(n)$ if $\{n\}\in U$, and
 $\tht_U(K(n))=0$ in all other cases.  In other words, with $\kp$ as
 in Remark~\ref{rem-kp} we have $\kp(U)=\{n\st\tht_U(K(n))\neq 0\}$.
 In particular, if $\kp(U)\neq\kp(V)$ then $\tht_U\not\simeq\tht_V$.
 However, it is common for $\kp(U)$ to be empty, so this is not a very
 strong result.
\end{remark}

\begin{lemma}\lbl{lem-tht-u}
 If $A=\{a_1,\dotsc,a_r\}$ with $a_1<\dotsb<a_r$ then there are
 equivalences of anafunctors
 \[ \tht_{uA} \simeq \phi_A \simeq L_{K(a_1)}\dotsb L_{K(a_r)}.  \]
\end{lemma}
\begin{proof}
 We have a diagram as follows, which commutes on the nose:
 \begin{center}
  \begin{tikzcd}[sep=huge]
   \CF
    \arrow[r,"i_{uA}^*"]
    \arrow[d,"j^*"',"\simeq"]
    \arrow[dr,"u^*","\simeq"'] &
   \CC \\
   \CC &
   \CP
    \arrow[l,"i^*_{(\emptyset,\emptyset)}","\simeq"']
    \arrow[u,"i_A^*"']
  \end{tikzcd}
 \end{center}
 This gives an equivalence
 $i_{uA}^*(j^*)^{-1}\simeq i_A^*(i_{(\emptyset,\emptyset)}^*)^{-1}$ of
 anafunctors, or in other words $\tht_{uA}\simeq\phi_A$.  Moreover,
 Corollary~\ref{cor-phi-factor} gives
 $\phi_A\simeq L_{K(a_1)}\dotsb L_{K(a_r)}$.
\end{proof}

\begin{lemma}\lbl{lem-tht-empty}
 The functor $\tht_\emptyset$ is zero.  
\end{lemma}
\begin{proof}
 Fix $X\in\CC(R)$ and choose $Y\in\CF(R)$ together with an isomorphism
 $X\simeq j^*(Y)$.  It will suffice to prove that $Y_\emptyset=0$.
 The $u$-cartesian property of $Y$ allows us to write $Y_\emptyset$ as
 a homotopy limit over the comma poset $\emptyset/u$.  However,
 $\emptyset$ is strictly larger than everything in the image of $u$
 (with respect to the reversed inclusion order that we are using on
 $\QQ$).  Thus, this comma poset is empty and the homotopy limit is
 zero as required.
\end{proof}

\begin{proposition}\lbl{prop-tht-lm}
 For $vA=\{T\sse N\st T\cap A\neq\emptyset\}$ we have
 $\tht_{vA}\simeq\lm_A$. 
\end{proposition}
\begin{proof}
 Let $Y$ be any object of $\CF(R)$.  We claim that the morphism
 $Y_{u\emptyset}\to Y_{vA}$ is a localisation with respect to $K(A)$.
 In order to simplify notation, we replace $\CC$ by $\CC^R$ and thus
 reduce to the case $R=1$.  As $Y$ is a $u$-cartesian object, we see
 that $Y_{vA}$ is the homotopy inverse limit of $(u^*Y)|_{vA}$.  Let
 $P'$ be the poset of nonempty subsets of $A$.  Note that the
 inclusion $i\:P'\to vA$ is left adjoint to the map $r\:vA\to P'$
 given by $rT=T\cap A$.  It follows from
 Proposition~\ref{prop-cofinal} that $i$ is homotopy cofinal, so
 $Y_{vA}$ is also the homotopy inverse limit of $(u^*Y)|_{P'}$.  We
 can now apply Proposition~\ref{prop-fracture} (with $N$ replaced by
 $A$) to see that this homotopy limit is a $K(A)$-localisation, as
 required.

 Now define $k\:[1]\to\QQ$ by $k(0)=u\emptyset$ and $k(1)=vA$.  This
 gives a morphism $k^*\:\CC^{\QQ}\to\CC^{[1]}$, and the previous
 paragraph shows that this restricts to give a morphism $\CF\to\CL$
 (where $\CL$ is as in Definition~\ref{defn-deriv-loc}, for
 localisation with respect to $K(A)$).  We now have a diagram as
 follows, which commutes on the nose:
 \begin{center}
  \begin{tikzcd}[sep=huge]
   \CF
    \arrow[r,"i_{vA}^*"]
    \arrow[d,"j^*"',"\simeq"]
    \arrow[dr,"k^*","\simeq"'] &
   \CC \\
   \CC &
   \CL
    \arrow[l,"i_0^*","\simeq"']
    \arrow[u,"i_1^*"']
  \end{tikzcd}
 \end{center}
 As $i_0^*$ and $j^*$ are equivalences, we see that $k^*$ is also an
 equivalence.   This gives an isomorphism
 $i_{vA}^*(j^*)^{-1}\simeq i_1^*(i_0^*)^{-1}$ of
 anafunctors, or in other words $\tht_{vA}\simeq\lm_A$.
\end{proof}

We now want to prove the following result:
\begin{theorem}\lbl{thm-tht-tht}
 The composite $\tht_U\tht_V$ is naturally isomorphic to $\tht_{U*V}$.
\end{theorem}
The proof will be given after some preliminaries.

We start with the following result, which will be needed in the proof
of Theorem~\ref{thm-tht-tht}, and which also shows that
Theorem~\ref{thm-tht-tht} is consistent with
Proposition~\ref{prop-phi-comp}. 
\begin{lemma}\lbl{lem-mu-u}
 For $A,B\in\CP$ we have
 \[ uA * uB =
     \begin{cases}
      u(A\cup B) & \text{ if } A \angle B \\
      \emptyset  & \text{ otherwise. }
     \end{cases}
 \]
\end{lemma}
\begin{proof}
 By definition, we have
 \[ uA * uB = \{ C \cup D \st A\sse C, B\sse D, C\angle D \} \sse
     u(A\cup B).
 \]
 If $A\angle B$ then we can choose $k$ with $a\leq k$ for all
 $a\in A$, and $k\leq b$ for all $b\in B$.  Then any
 $E\in u(A\cup B)$ can be written as $C\cup D$ with
 $C=\{j\in E\st j\leq k\}\supseteq A$ and
 $D=\{j\in E\st j\geq k\}\supseteq B$, so $E\in uA * uB$.  We
 therefore have $uA * uB = u(A\cup B)$ in this case.  On the other
 hand, if it is not true that $A\angle B$ then we can choose $a\in A$
 and $b\in B$ with $a>b$.  If $C$ and $D$ are as in the definition
 then $a\in C$ and $b\in D$ so it is not true that $C\angle D$.  From
 this it follows that $uA * uB=\emptyset$.
\end{proof}

\begin{definition}\lbl{defn-double-fracture}
 We say that an object $X\in\CC(\QQ\tm\QQ\tm R)$ is a
 \emph{double fracture object} if
 \begin{itemize}
  \item[(a)] $X$ is a fracture object relative to $\QQ\tm R$, so it
   lies in $\CF(\QQ\tm R)$.
  \item[(b)] The restriction to
   $\{u\emptyset\}\tm\QQ\tm R\simeq\QQ\tm R$ lies in $\CF(R)$.
 \end{itemize}
 We write $\CF_2(R)$ for the full subcategory of double fracture
 objects in $\CC(\QQ\tm\QQ\tm R)$.  We also define
 $k,l\:\QQ\to\QQ\tm\QQ$ by $k(V)=(u\emptyset,V)$ and
 $l(U)=(U,u\emptyset)$.  This gives functors
 $k^*,l^*\:\CF_2(R)\to\CC(\QQ\tm R)$.  Finally, we define
 $j_2\:e\to\QQ\tm\QQ$ to be the map with image
 $(u\emptyset,u\emptyset)$.  
\end{definition}

\begin{proposition}\lbl{prop-double-fract}
 $\CF_2$ is a thick subderivator of $\CC^{\QQ\tm\QQ}$, and the
 maps $k$ and $l$ give equivalences as shown:
 \begin{center}
  \begin{tikzcd}[sep=huge]
   \CF_2 \arrow[rr,"k^*","\simeq"'] \arrow[d,"l^*"',"\simeq"]
     \arrow[drr,"j_2^*","\simeq"'] & &
   \CF \arrow[d,"j^*","\simeq"'] \\
   \CF \arrow[rr,"j^*"',"\simeq"] & &
   \CC.
  \end{tikzcd}
 \end{center}
\end{proposition}
\begin{proof}
 Put $\CE(R)=\CF(\QQ\tm R)\subset\CC(\QQ\tm\QQ\tm R)$ (so this is the
 subcategory of objects satisfying condition~(a)).  From
 Proposition~\ref{prop-fracture} we see that $\CE$ is a thick
 subderivator of $\CC^{\QQ\tm\QQ}$ and that $k^*\:\CE\to\CC^{\QQ}$ is
 an equivalence of derivators.  As $\CF$ is a thick subderivator of
 $\CC^{\QQ}$, it follows that the preimage under $k^*$ of $\CF$ is a
 thick subderivator of $\CE$.  However, this preimage is just
 $\CF_2$, so $\CF_2$ is a thick subderivator as claimed.   It is also
 clear from this that $k^*\:\CF_2\to\CF$ is an equivalence.  We have
 seen that $j^*\:\CF\to\CC$ is also an equivalence.

 Next, recall again that $\CF$ is a thick subderivator.  Any
 monotone map $f\:R\to R'$ gives a functor
 $(1\tm f)^*\:\CC(\QQ\tm R')\to\CC(\QQ\tm R)$, and the subderivator
 property implies that $(1\tm f)^*(\CF(R'))\sse\CF(R)$.  Take
 $R'=\QQ\tm R$ and $f(r)=(u\emptyset,r)$; the conclusion is then that
 $l^*(\CE(R))\sse\CF(R)$, and so $l^*(\CF_2(R))\sse\CF(R)$.  This
 means that we have a diagram of functors as claimed, commuting up to
 natural isomorphism.  As $j^*$ and $k^*$ are equivalences, we can
 chase the diagram to see that $l^*$ and
 $i_{(u\emptyset,u\emptyset)}^*$ are equivalences as well.
\end{proof}

\begin{corollary}\lbl{cor-tht-tht}
 For any $U,V\in\QQ$, the composite anafunctor $\tht_U\tht_V$ is
 isomorphic to the fraction
 \[ \CC \xra{(j_2^*)^{-1}} \CF_2 \xra{i_{(U,V)}^*} \CC \]
\end{corollary}
\begin{proof}
 Note that if $X\in\CF_2(R)$ then $X\in\CF(\QQ\tm R)$ and $\CF$ is a
 subderivator so we have $i_V^*X\in\CF(R)$.  We can thus interpret
 $i_V^*$ as a morphism from $\CF_2$ to $\CF$.  It fits into a diagram
 as follows, which commutes on the nose:
 \begin{center}
  \begin{tikzcd}[sep=huge]
   && \CC \\ &
   \CF_2
    \arrow[ur,"i_{(U,V)}^*"]
    \arrow[r,"i_V^*"']
    \arrow[d,"k^*","\simeq"']
    \arrow[dl,"j_2^*"',"\simeq"] &
   \CF
    \arrow[u,"i_U^*"']
    \arrow[d,"j^*","\simeq"'] \\
   \CC &
   \CF
    \arrow[l,"j^*","\simeq"']
    \arrow[r,"i_V^*"'] &
   \CC
  \end{tikzcd}
 \end{center}
 The bottom edge represents the anafunctor $\tht_V$, whereas the right
 hand edge represents $\tht_U$.  The claim is clear from this.
\end{proof}

\begin{proposition}\lbl{prop-uu}
 The morphism $u^2_*\:\CC^{\PP\tm\PP}\to\CC^{\QQ\tm\QQ}$ restricts to
 give an equivalence $\CP_2\to\CF_2$, with inverse $(u^2)^*$.
\end{proposition}
\begin{proof}

 Before starting we warn the reader that in this proof all the
 restrictions will be with respect the base derivator $\CC$, even if
 we will apply them to elements we will prove are in the derivator of
 (doubly) localizing or fracture objects. This lets us avoid awkward
 notation and it is not restricting at all since the above derivators
 are subderivators of appropriate shifts of $\CC$.

 We must show that $\CP_2(R)\simeq\CF_2(R)$ for all $R$, but we can
 reduce to the case $R=e$ by replacing $\CC$ with $\CC^R$.

 We will factor the map
 $u^2\:\PP\tm\PP\to\QQ\tm\QQ$ as $u^2=u_1\circ u_2$, where
 $u_1=u\tm 1\:\PP\tm\QQ\to\QQ\tm\QQ$ and
 $u_2=1\tm u\:\PP\tm\PP\to\PP\tm\QQ$.  We also use the maps
 $i_\emptyset\:\PP\to\PP\tm\PP$ and $i_{u\emptyset}\:\QQ\to\QQ\tm\QQ$
 given by $i_{\emptyset}(B)=(\emptyset,B)$ and
 $i_{u\emptyset}(V)=(u\emptyset,V)$.  These fit in a commutative
 diagram 
 \begin{center}
  \begin{tikzcd}[sep=large]
   \QQ\tm\QQ &
   \QQ \arrow[l,"i_{u\emptyset}"'] \\
   \PP\tm\PP \arrow[u,"u^2"] &
   \PP \arrow[l,"i_\emptyset"] \arrow[u,"u"']
  \end{tikzcd}
 \end{center}
  
 Note that an object $X\in\CC(\PP\tm\PP)$ lies in $\CP_2(e)$ iff it
 satisfies the following conditions:
 \begin{itemize}
  \item[(a)] $X\in\CP(\PP)$
  \item[(b)] $i_\emptyset^*X\in\CP(e)$.
 \end{itemize}
 
 Similarly, by unwinding the definitions a little we see that an
 object $Y\in\CC(\QQ\tm\QQ)$ lies in $\CF_2(e)$ iff the following hold:
 \begin{itemize}
  \item[(c)] $u_1^*Y\in\CP(\QQ)$
  \item[(d)] $Y=(u_1)_*(u_1^*Y)$
  \item[(e)] $i_{u\emptyset}^*Y\in\CF(e)$.
 \end{itemize}

 Suppose that $Y\in\CF_2(e)$, so that~(c), (d) and~(e) are satisfied.
 Put $X=(u^2)^*Y\in\CC(\PP\tm\PP)$; we must show that $X\in\CP_2(e)$,
 or in other words that~(a) and~(b) are satisfied.  Note that
 $X=u_2^*(u_1^*Y)$ and $u_1^*Y\in\CP(\QQ)$ by~(c) and $\CP$ is a
 subderivator so $u_2^*(u_1^*Y)\in\CP(\PP)$ so~(a) is satisfied. 
 Moreover, the diagram shows that
 $i_\emptyset^*X=i_\emptyset^*(u^2)^*Y=u^*i_{u\emptyset}^*Y$, and
 $i_{u\emptyset}^*Y\in\CF(e)$ by~(e), so
 $u^*i_{u\emptyset}^*Y\in\CP(e)$, so~(b) holds.

 Suppose instead that we start with $X\in\CP_2(e)$, so that~(a) and~(b)
 hold.  Put $Y=u^2_*X\in\CC(\QQ\tm\QQ)$; we must then prove~(c), (d)
 and~(e).  We first note that $Y=(u_1)_*(u_2)_*X$ and $u_1$ is an
 embedding so $u_1^*(u_1)_*\cong 1$ so $u_1^*Y\cong(u_2)_*X$.  Moreover, we
 have $X\in\CP(\PP)$ by~(a) and $\CP$ is a subderivator so
 $(u_2)_*X\in\CP(\QQ)$ and this proves~(c).  Condition~(d)
 is also clear from this discussion.  For condition~(e), note that the
 diagram gives a Beck-Chevalley transformation
 \[ \al \: i_{u\emptyset}^*Y = i_{u\emptyset}^*u^2_*X \to
      u_* i_\emptyset^* X \in\CC(\QQ).
 \]
 We know that $i_\emptyset^* X\in\CP(e)$ by~(b), and it follows that
 $u_* i_\emptyset^* X\in\CF(e)$.  For condition~(e) it will therefore
 suffice to show that $\al$ is an isomorphism.  For this it will in
 turn suffice to check that $i_V^*\al$ is an isomorphism in $\CC(e)$
 for all $V\in\QQ$.  Put
 \begin{align*}
  \mathbb{B} &= \{B\in\PP\st V\leq uB\} \\
  \mathbb{C} &= \{(A,B)\in\PP\tm\PP\st (u\emptyset,V)\leq(uA,uB)\}.
 \end{align*}
 (These can in fact be simplified to $\mathbb{B}=V$ and
 $\mathbb{C}=\PP\tm V$.)  The map $i_{\emptyset}$ restricts to give a
 map $\mathbb{B}\to\mathbb{C}$.  The Kan formula tells us that the
 domain of $i_V^*\al$ is $\holim_{\mathbb{C}}X$, whereas the codomain
 is $\holim_{\mathbb{B}}i_\emptyset^*X$.  The evident projection
 $\mathbb{C}\to\mathbb{B}$ is right adjoint to $i_\emptyset$, so
 $i_\emptyset\:\mathbb{B}\to\mathbb{C}$ is homotopy cofinal by
 Proposition~\ref{prop-cofinal}, so $\al$ is an isomorphism as
 required.

 We now have morphisms $u^2_*\:\CP_2\to\CF_2$ and
 $(u^2)^*\:\CF_2\to\CP_2$ with $(u^2)^*u^2_*\simeq 1$ by
 Lemma~\ref{lem-embedding}.  All that is left is to prove that when
 $Y\in\CF_2(e)$, the unit map $Y\to u^2_*(u^2)^*Y$ is an isomorphism.
 Put $Z=u_1^*Y$, so condition~(d) gives $Y=(u_1)_*Z$.  It will suffice
 to show that $Z=(u_2)_*u_2^*Z$.  Note that $Z\in\CP(\QQ)$ by
 condition~(c), and $\CP$ is a subderivator, so $(u_2)_*u_2^*Z$ also
 lies in $\CP(\QQ)$.  As $i_\emptyset^*\:\CP\to\CC$ is an equivalence,
 it will suffice to check that the map
 $i_\emptyset^*Z\to i_\emptyset^*(u_2)_*u_2^*Z$ is an isomorphism.
 For this, we claim that
 $i_\emptyset^*(u_2)_*u_2^*Z=u_*u^*i_\emptyset^*Z$.
 This can be checked using the Kan formula, or by recalling that
 $i_\emptyset^*\:\CC^{\PP}\to\CC$ is a morphism of derivators and so
 is compatible with $u_*$ and $u^*$.  We must therefore check that the
 map $i_\emptyset^*Z\to u_*u^*i_\emptyset^*Z$ is an isomorphism.  Here
 $i_\emptyset^*Z=i_\emptyset^*u_1^*Y=i_{u\emptyset}^*Y$, and this lies
 in $\CF(e)$ by condition~(e),  so the claim follows from the
 definition of $\CF$.
\end{proof}

\begin{definition}\lbl{defn-ostar}
 Given $U,V\in\QQ$ we put
 \[ U\ostar V = (U\tm V) \cap \MM =
     \{(A,B)\in\PP\tm\PP\st A\in U,\; B\in V,\; A\angle B\}.
 \]
 The definition of $U*V$ can then be written as
 \[ U * V = \{A\cup B\st (A,B)\in U\ostar V\}. \]
 Note that $U\ostar V$ and $U*V$ can be seen as subposets of $\MM$ and
 $\PP$ respectively.  We define $\sg\:U\ostar V\to U*V$ by
 $\sg(A,B)=A\cup B$, and note that this is a morphism of posets.
\end{definition}

\begin{proposition}\lbl{prop-sg-cofinal}
 The map $\sg\:U\ostar V\to U*V$ is homotopy cofinal.
\end{proposition}
\begin{proof}
 Consider an element $C\in U*V$ and the comma poset
 \[ \sg/C = \{(A,B)\in U\ostar V\st A\cup B\sse C\}. \]
 By Proposition~\ref{prop-cofinal}, it will be enough to show that
 this is strongly contractible.  For $-1\leq i\leq n^*$ we write
 $C_{\leq i}=\{c\in C\st c\leq i\}$ and similarly for $C_{\geq i}$.
 As $C\in U*V$ we can write $C=A_0\cup B_0$ for some $A_0\in U$ and
 $B_0\in V$ with $A_0\angle B_0$.  This means that we can choose $k$
 with $a\leq k$ for all $a\in A_0$, and $k\leq b$ for all $b\in B_0$.
 It follows that $C_{\leq k}\in U$ and $C_{\geq k}\in V$.  Let $i$ be
 least such that $C_{\leq i}\in U$, and let $j$ be largest such that
 $C_{\geq j}\in V$. Trivially $i \leq k \leq j$, thus
 $(C_{\leq i}, C_{\geq j}) \in \sigma/C$.  Now consider an arbitrary
 element $(A,B)\in\sg/C$.  We define
 \[ \phi(A,B) = (C_{\leq\max(A)},\;C_{\geq\min(B)}). \]
 We use the conventions $\max(\emptyset)=-1$ and $\min(\emptyset)=n^*$
 if necessary; this ensures that $\max(A)\leq\min(B)$ in all cases, so
 $C_{\leq\max(A)}\angle C_{\geq\min(B)}$.  It is also clear that
 $A\sse C_{\leq\max(A)}$, so $C_{\leq\max(A)}\in U$, and similarly
 $C_{\geq\min(B)}\in V$.  Thus, $\phi$ is a morphism of posets from
 $\sg/C$ to itself, with $\phi\geq 1$.  On the other hand, if we
 define $\psi\:\sg/C\to\sg/C$ to be the constant map with
 value $(C_{\leq i},C_{\geq j})$, we find that $\psi\leq\phi$.  This
 gives the required contraction of $\sg/C$.
\end{proof}

\begin{proposition}\lbl{prop-mu-star}
 Consider the map $\mu\:\QQ\tm\QQ\to\QQ$ (given by $(U,V)\mapsto U*V$)
 and the induced morphism $\mu^*\:\CC^{\QQ}\to\CC^{\QQ\tm\QQ}$.  This
 restricts to give an equivalence $\mu^*\:\CF\to\CF_2$, making the
 following diagram commute up to natural isomorphism:
 \begin{center}
  \begin{tikzcd}[sep=huge]
   \CP \arrow[r,"u_*","\simeq"'] \arrow[d,"\inc_*\sg^*"',"\simeq"] &
   \CF \arrow[d,"\mu^*","\simeq"'] \\
   \CP_2 \arrow[r,"u^2_*"',"\simeq"] &
   \CF_2
  \end{tikzcd}
 \end{center}
\end{proposition}
\begin{proof}
 By the usual reduction, it will suffice to work with $\CP(e)$,
 $\CF(e)$ and so on.  Proposition~\ref{prop-doubly-loc-alt} gives the
 left hand equivalence, the top and bottom equivalences are given by
 Proposition~\ref{prop-fracture-obj} and Proposition~\ref{prop-uu}
 respectively.  We claim that for $X\in\CP(e)$, there is a natural
 isomorphism $\mu^*u_*X\simeq u^2_*\inc_*\sg^*X$.  Assuming this,
 everything else follows easily by chasing the diagram.  To prove the
 claim, we apply Proposition~\ref{prop-cofinal-square} to the square
 \begin{center}
  \begin{tikzcd}
   \MM \arrow[r,"\sg"] \arrow[d,"u^2\circ\inc"'] & \PP \arrow[d,"u"] \\
   \QQ\tm\QQ \arrow[r,"\mu"'] & \QQ,
  \end{tikzcd}
 \end{center}
 which commutes by Lemma~\ref{lem-mu-u}.  The square gives a
 Beck-Chevalley transform $\al\:\mu^*u_*\to u^2_*\inc_*\sg^*$, and the
 proposition tells us that this is an isomorphism provided that the
 map
 \[ \sg_{(U,V)}\:(U,V)/(u^2\circ\inc) \to \mu(U,V)/u \]
 is homotopy cofinal for all $U,V\in\QQ$.  By unwinding the
 definitions, we see that this is just the map $U\ostar V\to U*V$
 whose cofinality was proved in Proposition~\ref{prop-sg-cofinal}.
\end{proof}

\begin{proof}[Proof of Theorem~\ref{thm-tht-tht}]
 Using Proposition~\ref{prop-mu-star} we obtain the following diagram,
 which commutes on the nose:
 \begin{center}
  \begin{tikzcd}[sep=huge]
   \CF
    \arrow[r,"i_{U*V}^*"]
    \arrow[dr,"\mu^*","\simeq"']
    \arrow[d,"j^*","\simeq"'] &
   \CC \\
   \CC &
   \CF_2
    \arrow[l,"j_2^*","\simeq"']
    \arrow[u,"i_{(U,V)}^*" right]
  \end{tikzcd}
 \end{center}
 One route around the square gives $\tht_{U*V}$ by definition, and the
 other gives $\tht_U\tht_V$ by Corollary~\ref{cor-tht-tht}.
\end{proof}

We conclude with an addendum to Proposition~\ref{prop-sg-cofinal}.  We
do not currently have any use for this, but the method of proof is
interesting so we have included it.

\begin{proposition}\lbl{prop-sg-final}
 The map $\sg\:U\ostar V\to U*V$ is also homotopy final. 
\end{proposition}

The proof will be given after some preliminaries.

\begin{definition}\lbl{defn-al-bt}
 We also define maps $\al_i,\bt_i\:\MM\to\MM$ as follows:
 \begin{align*}
  \al_{2i}(A,B)   &= (A_{<i}    ,A_{\geq i}\cup B) &
  \al_{2i+1}(A,B) &= (A_{\leq i},A_{\geq i}\cup B) \\
  \bt_{2i}(A,B)   &= (A\cup B_{<i}    ,B_{\geq i}) &
  \bt_{2i+1}(A,B) &= (A\cup B_{\leq i},B_{\geq i})
 \end{align*}
 Here $A_{>i}$ means $\{a\in A\st a>i\}$, and so on.  We note that
 \begin{align*}
  \al_{2i+2}(A,B) &= (A_{\leq i},A_{>i}\cup B) \\
  \bt_{2i+2}(A,B) &= (A\cup B_{\leq i},B_{>i}).
 \end{align*}
\end{definition}

\begin{lemma}\lbl{lem-al-bt}
 All the above maps $\al_i$ and $\bt_i$ are poset maps, with
 $\sg\al_i=\sg\bt_i=\sg\gm_i=\sg\dl_i=\sg$.  The extreme cases are
 \begin{align*}
  \al_0(A,B)      &= (\emptyset,A\cup B) &
  \al_{2n^*}(A,B) &= (A,B) \\
  \bt_0(A,B)      &= (A,B) &
  \bt_{2n^*}(A,B) &= (A\cup B,\emptyset).
 \end{align*}
 There are inequalities $\al_{2i}\leq\al_{2i+1}\geq\al_{2i+2}$ and
 $\bt_{2i}\leq\bt_{2i+1}\geq\bt_{2i+2}$.
\end{lemma}
\begin{proof}
 Straightforward from the definitions.
\end{proof}

\begin{proof}[Proof of Proposition~\ref{prop-sg-final}]
 Consider $C\in U*V$.  By Proposition~\ref{prop-cofinal}, it will
 suffice to prove that the poset
 \[ C/\sg = \{(A,B)\in U\ostar V \st C\sse A\cup B\} \]
 is strongly contractible.  As in the proof of
 Proposition~\ref{prop-sg-cofinal}, we can choose $k$ between $0$ and
 $n^*-1$ such that $C_{\leq k}\in U$ and $C_{\geq k}\in V$.  We claim
 that for all $i\geq 2k+1$, the map $\al_i\:\MM\to\MM$ preserves
 $C/\sg$.  To see this, suppose that $(A,B)\in C/\sg$, so $A\in U$ and
 $B\in V$ and $A\angle B$ and $A\cup B\supseteq C$.  We have
 $\al_i(A,B)=(A_{<u},A_{\geq v}\cup B)$ for some $u,v$ with $u>k$.  We
 have seen that $\al_i$ preserves $\MM$ with $\sg\al_i=\sg$, so the
 only point to check is that $A_{\ge v}\cup B\in V$ and $A_{<u}\in U$.
 The first of these is clear because $B\in V$ and $V$ is closed
 upwards.  The second is also clear if $A_{<u}=A$.  Suppose instead
 that $A_{<u}\neq A$, so there exists $a\in A$ with $a\geq u$.
 It follows that for $b\in B$ we have $b\geq a\geq u>k$, so
 $C_{\leq k}\cap B=\emptyset$.  However, we have
 $C_{\leq k}\sse C\sse A\cup B$ by assumption, so
 $C_{\leq k}\sse A_{\leq k}\sse A_{<u}$.  As $C_{\leq k}\in U$ and $U$
 is closed upwards, we see that $A_{<u}\in U$ as required.  A
 symmetrical argument shows that $\bt_i$ preserves $C/\sg$ for
 $i\leq 2k+1$.  As $\al_{2i}\leq\al_{2i+1}\geq\al_{2i+2}$ we see that
 $[\al_{2k+1}]=[\al_{2n^*}]=1$ in the strong homotopy category.
 Similarly, we have $[\bt_{2k+1}]=[\bt_0]=1$ and thus
 $[\al_{2k+1}\bt_{2k+1}]=1$ in the strong homotopy category.  However,
 it is not hard to see that
 \[ \al_{2k+1}\bt_{2k+1}(A,B) =
     ((A\cup B)_{\leq k},(A\cup B)_{\geq k}) \geq
     (C_{\leq k},C_{\geq k}),
 \]
 so $\al_{2k+1}\bt_{2k+1}$ is equivalent to the constant map with
 value $(C_{\leq k},C_{\geq k})$.
\end{proof}

\section{Monoidal structures}

\begin{definition}
 A \emph{symmetric monoidal derivator} $\CC$ is as specified
 in~\cite{gr:mda}*{Definition 2.4}.  We will not give the full
 details, and we will use notation corresponding to stable homotopy
 theory rather than derived algebra.  The key points are as follows:
 \begin{itemize}
  \item[(a)] Each category $\CC(P)$ has a symmetric monoidal
   structure, with unit denoted by $S$ and the monoidal product of $X$
   and $Y$ called the \emph{smash product} and denoted by $X\Smash Y$.
  \item[(b)] For each $u\:P\to Q$, the pullback functor
   $u^*\:\CC(Q)\to\CC(P)$ preserves smash products up to isomorphism.
 \end{itemize}
\end{definition}

\begin{remark}
 It will not typically be true that the natural morphism
 $u_!(X\Smash u^*(Y))\to u_!(X)\Smash Y$ is an isomorphism.
 Similarly, even if there exist function objects $F(X,Y)$ with
 $[W,F(X,Y)]\simeq [W\Smash X,Y]$, these will typically not satisfy
 $u^*F(X,Y)\simeq F(u^*X,u^*Y)$.  These kinds of properties are valid
 for derivators indexed by groups or groupoids, but posets are the
 other extreme from that.  
\end{remark}
For the rest of this section, we will assume that $\CC$ is a symmetric
monoidal derivator.  We will also assume that the homology theories
$K(n)$ have the property that $K(n)_*(X)=0$ implies
$K(n)_*(X\Smash Y)=0$ for all $Y$.  We will just give some simple
results about how the monoidal structure interacts with chromatic
fracture. 

\begin{proposition}\lbl{prop-tht-monoidal-aux}
 Suppose that $X,Y,Z\in\CF(R)$.  Then the natural map
 \[ \CC(\QQ\tm R)(X\Smash Y,Z) \to \CC(R)(j^*X\Smash j^*Y,j^*Z)
 \]
 is bijective.
\end{proposition}
\begin{proof}
 As $Z$ is $u$-cartesian, we have
 \[ \CC(\QQ\tm R)(X\Smash Y,Z) =
     \CC(\PP\tm R)(u^*(X\Smash Y),u^*Z).
 \]
 Now just apply Lemma~\ref{lem-fully-loc-alt}.
\end{proof}

We can use the above result to show that $\tht_U$ is lax monoidal, in
the appropriate sense for anafunctors.  In more detail, consider the
following diagrams:
\[ \CC(R)\tm\CC(R) \xla{j^*\tm j^*} \CF(R)\tm\CF(R)
     \xra{\Smash} \CC(\QQ\tm R) \xra{i_U^*} \CC(R)
\]
\[ \CC(R)\tm\CC(R) \xra{\Smash} \CC(R) \xla{j^*} \CF(R) \xra{i_U^*} \CC(R). \]
The leftward-pointing maps are equivalences, so we can invert them to
get two different anafunctors $\CC(R)\tm\CC(R)\to\CC(R)$.  Informally,
these are $(X,Y)\mapsto\tht_U(X)\Smash\tht_U(Y)$ and
$(X,Y)\mapsto\tht_U(X\Smash Y)$.  We need to provide a morphism
between these anafunctors.  For this, we introduce the category
\[ \CP(R) = \{(X,Y,Z,u)\st X,Y,Z\in\CF(R),\;
                u\:j^*(X)\Smash j^*(Y)\xra{\simeq}j^*(Z)\},
\]
and the diagram
\begin{center}
 \begin{tikzcd}
  \CP(R) \arrow[d,"l"'] \arrow[r,"r"] &
  \CF(R) \arrow[d,"\inc"] \\
  \CF(R)\tm\CF(R) \arrow[d,"j^*\tm j^*"'] \arrow[r,"\Smash"] &
  \CC(\QQ\tm R) \arrow[r,"i_U^*"] \arrow[d,"j^*"] &
  \CC(R) \\
  \CC(R)\tm\CC(R) \arrow[r,"\Smash"'] &
  \CC(R).
 \end{tikzcd}
\end{center}
Here $l(X,Y,Z,u)=(X,Y)$ and $r(X,Y,Z,u)=Z$.  Using the fact that
$j^*\:\CF(R)\to\CC(R)$ is an equivalence, we find that $l$ is also an
equivalence and the rectangle formed of the two squares is a homotopy
pullback.  Because of this, our two anafunctors can be described as
follows: we invert the equivalence $\CP(R)\to\CC(R)\tm\CC(R)$, then
take one of the two routes around the top square, then apply $i_U^*$.
Proposition~\ref{prop-tht-monoidal-aux} gives a natural map
$X\Smash Y\to Z$ for all $(X,Y,Z,u)\in\CP(R)$, or in other words, a
natural map between the two composites around the top square.  This
gives our required morphism of anafunctors.

\appendix

\section{Compactly generated triangulated categories}
\label{apx-comp-gen}

We next want some results about Brown representability and Bousfield
localisation in triangulated categories and derivators.  For the
homotopy category of spectra, all statements are well-known with very
classical proofs that rely on having an underlying geometric category
of spectra~\cite{ma:ssa}*{Chapter 7}.  There are also proofs of
similar results in more axiomatic frameworks, relying only on the
theory of triangulated categories.  These are typically formulated in
the context of well-generated categories as defined by
Neeman~\cite{ne:tc}, and the proofs are somewhat complex.  It is
well-known to experts that everything becomes much simpler, and much
closer to the original results for the category of spectra, if we
restrict attention to compactly generated categories.  However, it
seems surprisingly hard to find an full account of this in the
literature. We therefore provide one here.

\begin{definition}\lbl{defn-CT}
 Until further notice, $\CT$ will be a compactly generated
 triangulated category with coproducts.  We choose a small skeleton
 $\CT_0$ of $\CT_c$, and note that this is necessarily closed under
 suspensions and desuspensions and cofibres up to isomorphism.  We
 also choose an infinite cardinal $\kp_0$ such that the total number
 of morphisms in $\CT_0$ is at most $\kp_0$.  
\end{definition}

\begin{definition}\lbl{defn-CT-kp}
 Let $\kp$ be a cardinal that is at least as large as $\kp_0$.  We
 define subcategories $\CT_n^\kp\sse\CT$ as follows.  First, we let
 $\CT_0^\kp$ be the subcategory of objects that can be expressed as a
 coproduct $\bigoplus_{i\in I}T_i$, where $|I|\leq\kp$ and
 $T_i\in\CT_c$ for all $i$.  We then define $\CT^\kp_{n+1}$ to be the
 subcategory of objects $Z$ that can be expressed as the cofibre of a
 map from an object in $\CT^\kp_0$ to an object in $\CT^\kp_n$.
 Finally, we let $\CT^\kp_\infty$ be the subcategory of objects $X$
 that can be expressed as the telescope of a sequence $X_n$ with
 $X_n\in\CT^\kp_n$ for all $n$.  Given that $\CT_c$ is essentially
 small, we find that $\CT^\kp_n$ is essentially small for all
 $n\leq\infty$.  
\end{definition}

We now state a version of the Brown representability theorem:
\begin{theorem}\lbl{thm-brown}
 Let $K\:\CT^{\opp}\to\Ab$ be a cohomology theory (so $K$ converts all
 coproducts to products, and distinguished triangles to exact
 sequences).  Then $K$ is representable.
\end{theorem}
The basic method of proof is due to Brown, and similar axiomatic
versions have appeared with various different hypotheses in a number
of places such as~\cite{ma:ssa}*{Theorem 4.11}
and~\cite{hopast:ash}*{Theorem 2.3.2}.  There are also various
versions with weaker hypotheses and much more complicated proofs, such
as~\cite{ne:tc}*{Chapter 8}.  For completeness we will give
a brief account here with our current hypotheses and notation.
\begin{proof}
 We shall define recursively a sequence of objects
 \[ X(0) \xra{i_0} X(1) \xra{i_1} X(2)
         \xra{i_2} \ldots
 \]
 and elements $u(k)\in K(X(k))$ such that $i_k^*u(k+1)=u(k)$.  We
 start with
 \[ X(0) = \bigoplus_{Z\in\CT_0} \; \bigoplus_{v \in K(Z)} Z. \]
 We take $u(0)$ to be the element of
 \[ K(X(0)) = \prod_{Z\in\CT_0} \;
              \prod_{v \in K(Z)}  K(Z)
 \]
 whose $(Z,v)$th component is $v$.  We then set
 \[ T(k) =
    \{(Z,f) \st Z\in\CT_0,\;
              f \: Z \xra{} X(k),\;
              f^* u(k) = 0 \}.
 \]
 We define $X(k+1)$ by the cofiber sequence
 \[ \bigoplus_{(Z,f)\in T(k)} Z \xra{} X(k)
        \xra{i_k} X(k+1).
 \]
 By applying $K$ to this, we obtain a three-term exact sequence (with
 arrows reversed).  It is clear by construction that $u(k)$ maps to
 zero in the left hand term, so that there exists
 $u(k+1)\in K(X(k+1))$ with $i_k^*u(k+1)=u(k)$ as required.

 We now let $X$ be the telescope of the objects $X(k)$.
 The cofibration defining this telescope gives rise to a
 short exact sequence
 \[ 0 \xra{} {\invlim}{}^1_k K(\Sigma X(k)) \xra{} K(X)
      \xra{} \invlim_k K(X(k)) \xra{} 0 .
 \]
 Using this, we find an element $u\in K(X)$ that maps to $u(k)$ in
 each $K(X(k))$.  As in Yoneda's lemma, this induces a natural map
 $\tau_U\:[U,X]\xra{} K(U)$.  It is easy to see that $\tau_Z$ is an
 isomorphism for each $Z\in\CT_0$ (using the fact that
 these objects are small).  It is also easy to see that
 \[ \{Z\st\tau_{\Sigma^kZ}\text{ is an isomorphism for all $k$} \} \]
 is a localizing category.  It contains $\CT_0$, 
 so it must be all of $\CT$; thus $\tau$ is an isomorphism.
\end{proof}

Next, for any localising subcategory $\CU\sse\CT$, we can form the
Verdier quotient category $\CT/\CU$; see~\cite{ne:tc}*{Chapter 2} for
a detailed treatment.  However, we are implicitly assuming everywhere
that the hom sets of our categories are small, or in other words that
they are genuine sets rather than proper classes, and this can fail
for $\CT/\CU$.  If we can verify in a particular case that $\CT/\CU$
has small hom sets, then we can use Brown representability to prove
the existence of localisation functors.  This general line of argument
is well-known, going back to Adams and Bousfield.  There are various
known approaches to prove that $\CT/\CU$ has small hom sets.  One
possibility is to assume that $\CT$ is the homotopy category of a
Quillen model category or an infinity category in the sense of Lurie;
but here we prefer to work solely with triangulated categories and
derivators.  In~\cite{ma:ssa}*{Chapter 7}, Margolis proves some
results of this type for the category of spectra, and it is well-known
to experts that his approach can be generalised to compactly generated
triangulated categories.  As with the representability theorem, there
are also similar results with weaker hypotheses and much more
complicated proofs, such as~\cite{ne:tc}*{Corollary 4.4.3}.  However,
we have not been able to find an explicit axiomatic version of the
approach of Margolis in the literature, so we provide one here.  We
start by spelling out the argument that small hom sets give
localisations. 

\begin{theorem}\lbl{thm-loc}
 Suppose that $\CT/\CU$ has small hom sets.  Then for each $X\in\CT$
 there exists a distinguished triangle
 \[ CX \xra{q} X \xra{j} LX \xra{d} \Sg CX \]
 with $CX\in\CU$ and $LX\in\CU^\perp$.  In fact, $q$ is terminal in
 $\CU/X$ and $j$ is initial in $X/\CU^\perp$.
\end{theorem}
\begin{proof}
 From~\cite{ne:tc}*{Chapter 2} we have a triangulation of
 $\CT/\CU$ such that the evident functor $\pi\:\CC\to\CC/\CU$ is
 exact.  From~\cite{ne:tc}*{Corollary 3.2.11} we know that $\CT/\CU$
 has small coproducts and that $\pi$ preserves coproducts.  It follows
 that the functor $W\mapsto(\CT/\CU)(\pi W,\pi X)$ is a cohomology
 theory on $\CT$.  By Theorem~\ref{thm-brown}, we can choose an object
 $LX\in\CT$ and a natural isomorphism
 $\CT(W,LX)\simeq(\CT/\CU)(\pi W,\pi X)$ for all $W\in\CT$.  If
 $W\in\CU$ then $\pi W=0$ and so $\CT(W,LX)=0$; this proves that
 $LX\in{}^\perp\CU$.  The natural map
 \[ \CT(W,X) \xra{\pi} (\CT/\CU)(\pi W,\pi X) \simeq
     \CT(W,LX) 
 \]
 corresponds (via the Yoneda Lemma) to a map $j\:X\to LX$.  We can
 then fit this into a distinguished triangle
 \[ CX \xra{q} X \xra{j} LX \xra{d} \Sg CX. \]
 Next, recall that every element of $(\CT/\CU)(\pi W,\pi LX)$ can be
 represented as a fraction $gf^{-1}$ for some diagram
 $(W\xla{f}V\xra{g}LX)$ with $\cof(f)\in\CU$.  As $LX\in\CU^\perp$ we
 see that $\CT(\cof(f),LX)=\CT(\Sg\cof(f),LX)=0$, and it follows that
 there is a unique $h\:W\to LX$ with $hf=g$.  Using this, we find that
 the map
 \[ \pi \: \CT(W,LX) \to (\CT/\CU)(\pi W,\pi LX) \]
 is an isomorphism for all $W$.  By combining this with our
 isomorphism $\CT(W,LX)\simeq(\CT/\CU)(\pi W,\pi X)$, we see that
 $\pi(j)\:\pi X\to\pi LX$ is an isomorphism, so $\pi CX=0$ and
 $CX\in\CU$ as claimed.  Now for $U\in\CU$ we have $\CT(U,LX)_*=0$ hence
 the distinguished triangle gives $\CT(U,CX)\simeq\CT(U,X)$, which
 proves that $(CX\xra{q}X)$ is terminal in $\CU/X$.  Dually, if
 $Z\in\CU^\perp$ then $\CT(CX,Z)_*=0$ so the distinguished triangle
 gives $\CT(X,Z)\simeq\CT(LX,Z)$, therefore $(X\xra{j}LX)$ is initial in
 $X/\CU^\perp$. 
\end{proof}

\begin{definition}\lbl{defn-weakly-initial}
 Let $\CC$ be a category, and let $\CJ$ be a set of objects.  We say
 that $\CJ$ is \emph{weakly initial} if for every $X\in\CC$ there
 exists an object $T\in\CJ$ and a morphism $T\to X$.
\end{definition}

\begin{proposition}\lbl{prop-weakly-initial}
 Suppose that $\CU$ is a localising subcategory of $\CT$ such that for
 each $X\in\CT$, the comma category $X/\CU$ has a weakly initial set.
 Then the category $\CT/\CU$ has small hom sets.
\end{proposition}
\begin{remark}
 As we will explain in more detail below, this implies that there is a
 localisation functor $L$ with kernel $\CU$, and thus that the unit
 map $X\to LX$ is in initial in the comma category $X/\CU^\perp$
 (not $X/\CU$).  This interplay between $X/\CU^\perp$ and $X/\CU$ is
 a little surprising, but that is how the proof works.
\end{remark}
\begin{proof}
 Fix objects $X,Y\in\CT$, and let $\{X\xra{e_i}U_i\}_{i\in I}$ be a
 weakly initial family in $X/\CU$.  Let $Z_i\xra{f_i}X$ be the fibre
 of $e_i$.  As the cofiber of $f_i$ is in $\CU$, every map
 $g\:Z_i\to Y$ gives a fraction $gf_i^{-1}\in(\CT/\CU)(X,Y)$.  It will
 suffice to show that every element of $(\CT/\CU)(X,Y)$ is of this
 form.  A general element of $(\CT/\CU)(X,Y)$ can be represented as
 $qp^{-1}$ for some $X\xla{p}W\xra{q}Y$, where the cofibre of $p$ is in
 $\CU$.  In more detail, we have a cofibration $W\xra{p}X\xra{m}V$ with
 $V\in\CU$.  By the weakly initial property, we can write $m$ as
 $ne_i$ for some $i$ and some $n\:U_i\to V$.  This gives $mf_i=0$, so
 $f_i$ lifts to the fibre of $m$, so we can choose $k\:Z_i\to W$ with
 $pk=f_i$.  The cofibres of $p$ and $f_i$ are in $\CU$, and $\CU$ is
 thick, so the octahedral axiom tells us that the cofibre of $k$ is
 also in $\CU$, so $k$ becomes an isomorphism in $\CT/\CU$.  We can
 now put $g=qk\:Z_i\to Y$ and we have
 $qp^{-1}=(qk)(pk)^{-1}=gf_i^{-1}$ as required.
\end{proof}

\begin{proposition}\lbl{prop-homology-weakly-initial}
 Let $K\:\CT\to\Ab$ be a homology theory (so $K$ preserves all
 coproducts, and converts distinguished triangles to exact
 sequences).  Put
 \[ \CU=\ker(K_*)=
      \{X\st K(\Sg^nX)=0 \text{ for all } n\in\Z\}. \]
 Then $\CU$ is a localising subcategory of $\CT$ such that $X/\CU$ has
 a weakly initial set for all $X$, so $\CT/\CU$ has small hom sets.
\end{proposition}

The proof will be given after some preliminary definitions and
lemmas. 
\begin{definition}\lbl{defn-K-hat}
 We define $K_*\:\CT\to\Ab_*$ by $K_n(X)=K(\Sg^{-n}X)$.  We then
 define $\hK_*\:\CT\to\Ab_*$ to be the left Kan extension of the
 restriction of $K_*$ to $\CT_0\subset\CT$.  Explicitly, $\hK_*(X)$ is
 the colimit of the functor $\CT_0/X\to\Ab_*$ sending $(T\xra{t}X)$ to
 $K_*(T)$.  There is an evident counit map $\phi_X\:\hK_*(X)\to K_*(X)$.
\end{definition}

\begin{lemma}\lbl{lem-K-hat-K}
 The counit map $\phi_X\:\hK_*(X)\to K_*(X)$ is an isomorphism for all
 $X$.  
\end{lemma}
\begin{proof}
 It is not hard to see that the category $\CT_0/X$ is filtered,
 and to deduce that $\hK$ is again a homology theory.  Details
 are given in~\cite{hopast:ash}*{Section 2.3}, for example.  (In that
 reference there are officially some additional standing assumptions
 that we are not assuming here, such as that $\CT$ has a symmetric
 monoidal structure, but none of those assumptions are used in the
 relevant proofs.)  It follows that the category
 $\{X\st\phi_X \text{ is iso }\}$ is localising and contains $\CT_0$,
 so it must be all of $\CT$, as required.
\end{proof}

\begin{definition}\lbl{defn-kappa-one}
 We let $\kp_1$ be a cardinal such that $\kp_1\geq\kp_0$ and
 $\kp_1\geq|K_*(T)|$ for all $T\in\CT_0$.
\end{definition}

\begin{corollary}\lbl{cor-F-one}
 Fix a cardinal $\kp$ such that $\kp\geq\kp_1$.  Suppose that $X\in\CT$ with
 $|K_*(X)|\leq\kp$.  Then one can choose a fibre sequence
 \[ FX\xra{\al} X\xra{\bt} GX \]
 such that
 \begin{itemize}
  \item[(a)] $FX\in\CT_0^\kp$.
  \item[(b)] $K_*(\al)$ is surjective, so $K_*(\bt)=0$.
  \item[(c)] $|K_*(GX)|\leq\kp$.
 \end{itemize}
 (We do not claim that $F$ or $G$ is a functor.) 
\end{corollary}
\begin{proof}
 Let $\{b_i\st i\in I\}$ be a homogeneous generating set for $K_*(X)$
 with $|I|\leq\kp$.  As $K_*(X)=\hK_*(X)$, we can choose objects
 $T_i\in\CT_0$ and maps $t_i\:T_i\to X$ and elements $a_i\in K_*(T_i)$
 with $(t_i)_*(a_i)=b_i$.  We then take
 $FX=\bigoplus_iT_i\in\CT_0^\kp$, and let $\al\:FX\to X$ be the map
 given by $t_i$ on the $i$'th summand.  Using
 $K_*(FX)=\bigoplus_iK_*(T_i)$ we find that $K_*(\al)$ is surjective
 as required.  We then define $X\xra{\bt}GX$ to be the 
 cofibre of $\al$, so $K_*(\bt)=0$ and we have a short exact sequence
 $K_{*+1}(GX)\to K_*(FX)\to K_*(X)$.  Standard cardinal arithmetic now
 gives $|K_*(FX)|\leq\kp$ and then $|K_*(GX)|\leq\kp$. 
\end{proof}

\begin{construction}\label{cons-F-infty}
 We have a sequence $X\xra{\bt}GX\xra{\bt}G^2X\xra{}\dotsb$, and we
 define $G^\infty X$ to be the telescope.  For $0\leq n\leq\infty$ we
 define $F_nX$ to be the fibre of the map $X\to G^nX$.  The octahedral
 axiom then gives us a cofibration
 $\Sg^{-1}FG^nX\to F_nX\to F_{n+1}X$, and one can also check that
 $F_\infty X$ is the telescope of the objects $F_nX$ for $n<\infty$.
\end{construction}

\begin{proposition}\lbl{prop-F-infty}
 For $X$ and $\kp$ as above, the fibre sequence
 \[ F_\infty X \xrightarrow{\alpha_{\infty}} X \to G^\infty X \]
 satisfies
 \begin{itemize}
  \item[(a)] $F_\infty X\in\CT_\infty^\kp$.
  \item[(b)] $K_*(\al_\infty)$ is an isomorphism.
  \item[(c)] $K_*(G^\infty X)=0$, or in other words $G^\infty X\in\CU$.
 \end{itemize}
\end{proposition}
\begin{proof}
 We see by induction that $|K_*(G^nX)|\leq\kp$ for all $n$. It follows
 that $FG^nX\in\CT_0^\kp$, and thus that $F_nX\in\CT_{n-1}^\kp$, therefore $F_\infty X\in\CT_\infty^\kp$.  Also, as the maps
 $K_*(\bt)\:K_*(G^nX)\to K_*(G^{n+1}X)$ are zero, we see that
 $K_*(G^\infty X)=0$.  It follows that the map
 $K_*(F_\infty X)\to K_*(X)$ is an isomorphism as desired. 
\end{proof}

\begin{proof}[Proof of Proposition~\ref{prop-homology-weakly-initial}]
 It is straightforward to check that $\CU$ is a localising
 subcategory.

 Now fix $X\in\CT$, and choose $\kp\geq\kp_1$ such that
 $|K_*(X)|\leq\kp$.  Let $\CA$ be the subcategory of
 $X/\CU$ consisting of objects $(X\xra{f}U)$ such that the fibre of
 $f$ lies in $\CT^\kp_\infty$.   As $\CT_\infty^\kp$ is essentially
 small, we see that $\CA$ is also essentially small, so we can choose
 a small skeleton $\CA_0$.  

 Now let $(X\xra{g}V)$ is an arbitrary object of $X/\CU$.  Let
 $P\xra{j}X$ be the fibre of $g$, so $K_*(j)$ is an isomorphism, so
 $|K_*(P)|\leq\kp$.  Proposition~\ref{prop-F-infty} therefore gives us a
 map $Q=F_\infty P\xra{q}P$ such that $Q\in\CT_\infty^\kp$ and
 $K_*(q)$ is an isomorphism.  Let $X\xra{f}U$ be the cofibre of
 $jq\:Q\to X$.  As $K_*(jq)$ is an isomorphism, we see that $U\in\CU$
 and so $(X\xra{f}U)\in X/\CU$.  As $gj=0$ we have $g(jq)=0$ so the
 map $g\:X\to V$ factors through $f$, so there is a morphism from
 $(X\xra{f}U)$ to $(X\xra{g}V)$ in $X/\CU$.  This proves that $\CA$ is
 weakly initial in $X/\CU$, and it follows that the skeleton $\CA_0$
 has the same property.
\end{proof}

\end{document}